\documentclass[10pt,1p,times]{amsart}
\usepackage[T1]{fontenc}

\usepackage{amssymb}
\usepackage{amsmath}
\usepackage{amsthm}
\usepackage[foot]{amsaddr}

\usepackage{enumitem}  %fancy enumerations
\usepackage{enumitem} 	%resume/interupt enumerations
\usepackage{url}		%\url{}
\usepackage{framed}	%\begin{framed}..\end{framed}
\usepackage{hyperref}	%links within document
\usepackage{bbold}		%to get indicator function \mathbb{1}
\usepackage{color}		%\color{} command
\usepackage{graphicx}

\usepackage{chngcntr}

\theoremstyle{definition}
\newtheorem{definition}{Definition}[section]
\newtheorem{example}[definition]{\textit{Example}}
\newtheorem{theorem}[definition]{Theorem}	
\newtheorem{lemma}[definition]{Lemma}
\newtheorem{corollary}[definition]{Corollary}
\newtheorem{proposition}[definition]{Proposition}
\newtheorem{remark}[definition]{Remark}
\newcounter{appcount}

\newtheorem{applemma}[appcount]{Lemma}

\numberwithin{equation}{section}
\newcommand{\Bo}{\mathcal{B}}

\newcommand{\C}{\mathbb{C}}
\newcommand{\Cp}{\C_{+}}
\newcommand{\Cm}{\C_{-}}
\newcommand{\N}{\mathbb{N}}
\newcommand{\R}{\mathbb{R}}
\newcommand{\Rp}{\R_{+}}

\newcommand{\ide}{\operatorname{I}}
\newcommand{\Semi}{$C_0$-semigroup }
\newcommand{\DA}{D(A)}

\newcommand{\veps}{\varepsilon}

\newcommand{\Hinf}{\mathcal{H}^{\infty}(\Cm)}

\newcommand{\Sect}{\operatorname{Sect}}

\renewcommand{\Hinf}{{H}^{\infty}}
\newcommand{\HinfO}{{H}_{0}^{\infty}}

\begin{document}

%\begin{frontmatter}
\title[On measuring the unboundedness of the $H^{\infty}$-calculus] {On measuring unboundedness of the $H^{\infty}$-calculus for generators of analytic semigroups}
\author{Felix L. Schwenninger}
\email{felix.schwenninger@uni-hamburg.de}
%\fntext[fn1]{\textit{corresponding author, telephone:}+31-53-489-3457.}

\address{Dept.\ of Applied Mathematics, University of Twente, The Netherlands}
%\fntext[fn1]{Present address: School of Mathematics and Natural Sciences, Arbeitsgruppe Funktionalanalysis, University of Wuppertal, 42119 Wuppertal, Germany.}

\begin{abstract}

We investigate the boundedness of the $H^{\infty}$-calculus by
estimating the bound $b(\varepsilon)$ of the mapping
$H^{\infty}\rightarrow \mathcal{B}(X)$: $f\mapsto f(A)T(\varepsilon)$ for
$\varepsilon$ near zero.  Here, $-A$ generates the analytic semigroup
$T$ and  $H^{\infty}$ is the space of bounded analytic functions on a
domain strictly containing the spectrum of $A$.  We show that
$b(\varepsilon)=\mathcal{O}(|\log\varepsilon|)$ in general, whereas
$b(\varepsilon)=\mathcal{O}(1)$ for bounded calculi. 
This generalizes a result by Vitse and complements work by Haase and Rozendaal for non-analytic semigroups.
We discuss the sharpness of our bounds and show that single square function estimates yield $b(\varepsilon)=\mathcal{O}(\sqrt{|\log\varepsilon|})$.
\end{abstract}
\small
\maketitle

%\end{frontmatter}

\section{Introduction}
	
Functional calculus, the procedure to define a new operator as
evaluation of an intial operator in a (scalar-valued) function, had
its beginnings with von Neumann's work \cite{vonneumann} more than 80
years ago. Typically, the aim is to preserve the algebraic structures
of the set of functions for the operators, such as linearity and
multiplicativity.  Therefore, an ultimate goal is to get a
homomorphism from a function algebra to an operator algebra, e.g.\ the
Banach algebra of bounded operators on a Banach space.  However,
sometimes such a mapping is not possible for the chosen pair of
algebras and we are forced to weaken the homomorphism property. This
can be done by considering a subclass of functions first, on which a
homomorphism is possible, and extend this mapping (algebraically), see
e.g.\ \cite[Chapter 1]{haasesectorial} and the references therein.%  

In the case of the $\Hinf$-calculus this means that we may get
unbounded operators.  Here, we consider the pair of sectorial
operators $A$ and functions $f$ which are bounded and analytic on a
sector that contains the spectrum of $A$, 
see Section \ref{notions} for a brief introduction. 
 From the very beginnings of this calculus 30 years ago, \cite{mcintoshHinf},
 it has been known that
we cannot expect the $\Hinf$-calculus to be bounded, i.e., that $f(A)$
is a bounded operator for every $f\in\Hinf$, \cite{mcintoshHinfNo}.
Starting with the work by McIntosh, \cite{mcintoshHinf}, for sectorial
operators on Hilbert spaces, the $\Hinf$-calculus turned out to be
very useful in various situations, in particular for the study maximal
regularity, see \cite[Chapter 9]{haasesectorial}, \cite{KunstmannWeis04} and the references
therein. For a recent survey and open problems of the $H^{\infty}$-calculus for sectorial operators we refer to \cite{Fackler14Reg}.

The question of boundedness of the calculus in a particular situation
remains crucial in the applications and has been subject to research over the last decades, see e.g.\ \cite{cowlingdoustmcintoshyagi,KaltonWeis01,KunstmannWeis04} and \cite[Chapter 5]{haasesectorial} for an overview.
The main goal of this work is to
investigate and `measure' the (un)boundedness of the $\Hinf$-calculus.
\newline
Functional calculus for subalgebras of $H^{\infty}$ are of interest in their own right. 
For instance, in \cite{Vitse05} Vitse proves estimates for a  Besov space functional calculus for analytic semigroups, (see  \cite{haasetransference11} for the case of $C_{0}$-semigroup generators on Hilbert spaces). We will discuss this result in Section \ref{discussion} and give a slight improvement. 
Furthermore, the corresponding framework of $H^{\infty}$-calculus for $C_{0}$-semigroup generators was recently developed in \cite{battyhaasemubeen,haasehalfplaneoperators,mubeenPhD} where \textit{half-plane operators} take over the role of sectorial operators. \newline
Let us state a first observation which can be seen as the starting
point for the results to come.  For the precise definition of the used
notions and a proof we refer to Section \ref{notions} and Proposition \ref{prop3}.
\begin{proposition}\label{prop1}%
  Let $A$ be a densely defined, invertible, sectorial operator of
  angle $\omega<\frac{\pi}{2}$ on the Banach space $X$.  Then, for
  $\phi\in(\omega,\pi)$ the $\Hinf(\Sigma_{\phi})$-calculus is bounded
  if and only if
\begin{equation}\label{prop:bddcalceq1}
	\forall f\in\Hinf(\Sigma_{\phi}) \ \quad \limsup_{\veps\to0^{+}}\|(fe_{\veps})(A)\|=:C_{f}<\infty,
\end{equation}
where $e_{\veps}(z)=e^{-\veps z}$ and $\Sigma_{\phi}:=\left\{z\in\C:z\neq0,|\arg(z)|<\phi\right\}$.
\end{proposition}
  In Example \ref{ex:invertibility}, we show that the assumption
  of $A$ being invertible is needed to guarantee that $(fe_{\veps})(A)$ is a bounded operator for $\veps>0$.
 On the other hand, if we allow for $\omega=\frac{\pi}{2}$, then $(fe_{\veps})(A)$ can be unbounded, since $e_{\veps}$ cannot control the behavior of $f$ along the imaginary axis.
  However, it is a remarkable result that by incorporating the geometry of the Banach space, one indeed gets that $(fe_{\veps})(A)$ is bounded for, not necessarily analytic, $C_{0}$-semigroup generators $-A$ (which are sectorial operators of angle $\frac{\pi}{2}$). More precisely, on Hilbert spaces $(fe_{\veps})(A)$ always defines a
  bounded operator if $-A$ generates an exponentially stable semigroup
  and if $f$ is bounded and analytic on the right half-plane. This was first proved by Zwart in 
  \cite[Thm.~2.5]{ZwartAdmissible}. Using powerful \textit{transference principles} from \cite{haasetransference11},  Haase and Rozendaal generalized this to arbitrary Banach spaces for $f$ in the\textit{ analytic multiplier algebra} $\mathcal{AM}_{p}(X)\subset\Hinf(\Cp)$, $p\geq1$, see in \cite{HaaseRozendaal13}. Note that the latter inclusion is a strict embedding unless $p=2$ and $X$ is a Hilbert space (in which case equality holds by Plancherel's theorem). They also showed that, alternatively,  one can make additional assumptions on the semigroup rather than on the function space. Namely, by requiring that the (rescaled) semigroup is $\gamma$-bounded, see \cite[Thm.~6.2]{HaaseRozendaal13}. Again, this result generalizes the Hilbert space case as $\gamma$-boundedness coincides with classical boundedness then. Moreover, although norm bounds in terms of $\veps$ were already present in \cite{ZwartAdmissible}, they were significantly improved in \cite{HaaseRozendaal13},  see also below. 
We  remark that the definition of functional calculus for non-analytic $C_{0}$-semigroups differs by nature from the one for sectorial operators. Using the axiomatics of holomorphic calculus in \cite[Chapter 1]{haasesectorial}, this can be done by either directly extending the well-known Hille-Phillips calculus, see \cite{HaaseRozendaal13}, or the above-mentioned  calculus for half-plane operators, \cite{battyhaasemubeen,haasehalfplaneoperators,mubeenPhD}. In \cite{SchweZwa2012,ZwartAdmissible} an alternative definition using notions from systems theory is used. However, as all these techniques are extensions of the Hille-Phillips calculus, the notions are consistent in the considered situation. 
	
\mbox{}From Proposition \ref{prop1} we see that the behavior of the norm $\|(fe_{\veps})(A)\|$
for $\veps$ near zero characterizes the boundedness of the
$\Hinf$-calculus for the sectorial operator $A$ of angle less than
$\frac{\pi}{2}$ that has $0$ in its resolvent set.  The negative,
$-A$, of such an operator corresponds to the generator of an analytic
and exponentially stable $C_{0}$-semigroup $T$. By observing that $T(\veps)=e_{\veps}(A)$, we derive that
$(fe_{\veps})(A)=f(A)T(\veps)$ for $\veps>0$.  As the $\Hinf$-calculus need not be
bounded, in general, we cannot bound $\|(fe_{\veps})(A)\|$ uniformly in
$\veps$.  Therefore, it is our goal to establish estimates of the form
\begin{equation}\label{eq:int1}
			\|(fe_{\veps})(A)\| \leq b(\veps)\cdot \|f\|_{\infty},
\end{equation}
for all $f\in\Hinf$ on a sector larger than the sector of sectorality
of $A$.  In general, $b(\veps)$ will become unbounded for
$\veps\to0^{+}$.

In Theorem \ref{thm:expstab} we show that
$b(\veps)=\mathcal{O}(|\log\veps|)$ as $\veps\to0^{+}$ on general
Banach spaces.  For $0\notin\rho(A)$, we derive a similar result for
functions $f\in\Hinf$ which are holomorphic at $0$, see Theorem
\ref{thm1}.  It turns out that the latter result generalizes a result
by Vitse in \cite{Vitse05} and improves the dependence on the sectorality constant $M(A,\phi)$ significantly,
see Section \ref{sec:Vitse}.
Moreover, our techniques seem to be more elementary as we do not employ the Hille-Phillips calculus. 
\\
For Hilbert spaces and general exponentially stable
$C_0$-semigroup generators $-A$ an estimate of the form (\ref{eq:int1})
$b(\veps)=\mathcal{O}(\veps^{-\frac{1}{2}})$ was derived in
\cite{ZwartAdmissible}. It was subsequently improved to
$b(\veps)=\mathcal{O}(|\log\veps|)$ by Haase and Rozendaal,
\cite[Theorem 3.3]{HaaseRozendaal13}, using an adaption of a lemma due to Haase, Hyt\"onen, \cite[Lem.~A.1]{haasetransference11}.
As mentioned in the lines following Proposition \ref{prop1} above, the techniques rely on the
geometry of the Hilbert space and cannot be extended to
general Banach spaces without either changing to another function space, \cite[Thms.~3.3 and 5.1]{HaaseRozendaal13},  or strengthening the assumption on the semigroup using $\gamma$-boundedness, \cite[Thm.~6.2]{HaaseRozendaal13}.
Hence, our results can be seen as additionally requiring
analyticity of the semigroup, but dropping any additional assumption on the Banach space.
As will be visible in the proofs of Theorems \ref{thm1} and \ref{thm:expstab}, the logarithmic dependence on $\veps$ is more elementary to derive than for general semigroups.

Let us remark that estimates of the form (\ref{eq:int1}) reveal information about the domain of $f(A)$. 
In particular, $b(\veps)=\mathcal{O}(|\log\veps|)$ implies that $D(A^{\alpha})\subset D(f(A))$ for $\alpha>0$, see \cite[Thm.~3.7]{HaaseRozendaal13}. For instance, this can be used to derive convergence results for numerical schemes, see, e.g., \cite{EgertRozendaal13}.

In Section \ref{sec:diagop}, we show that the logarithmic behavior is essentially optimal on
Hilbert spaces by means of a scale of examples of Schauder basis multipliers. More precisely,
Theorem \ref{sharpnesslog} states that for any $\gamma<1$, there exists a
 sectorial operator on $L^{2}(-\pi,\pi)$ such that $b(\veps)$
grows like $|\log(\veps)|^{\gamma}$. In the examples we also focus on tracking the dependence on the sectorality constant.

\textit{Square function estimates} or \textit{quadratic estimates}
play a crucial role in characterizing bounded $\Hinf$-calculi for
sectorial operators, see \cite{cowlingdoustmcintoshyagi,Yakubovich2011,KaltonWeis01,KunstmannWeis04,mcintoshHinf}.
On Hilbert spaces this means that for some function $g\in\Hinf$ an estimate of the form
\begin{equation*}
		\int_{0}^{\infty} \|g(tA)x\|^{2}\frac{dt}{t}\leq K^{2} \|x\|^{2},\quad \forall x\in X,
\end{equation*} 
has to hold and an analogous one for the adjoint $A^{*}$.  Whereas it
is known that such an estimate for only one of $A$ or $A^{*}$ is not
sufficient for a bounded calculus, as shown  by Le Merdy in \cite{LeMerdy2003}, we show in Section
\ref{sec:Squarefunct} that a single estimate does improve the
situation in the way that $b(\veps)=\mathcal{O}(\sqrt{|\log\veps|})$
then.  Again, by means of an example it is shown that this behavior is
essentially sharp.
\newline
In Section \ref{discussion} we compare our result with the one by Haase, Rozendaal in the case of an analytic semigroup on a Hilbert space. Furthermore, using the results of Section \ref{sec:mainresults}, we derive a slightly improved estimate for the Besov space functional calculus introduced by Vitse in \cite{Vitse05}. We conclude by mentioning the relation to \textit{Tadmor--Ritt} or \textit{Ritt} operators which can be seen as the discrete analog for analytic semigroups.
%%%%%%%%%%%%%%%%%%%%%%%

\subsection{Semigroups, sectorial operators and functional
  calculus}\label{notions}

In the following let $X$ denote a complex Banach space.  If $X$ is a
Hilbert space the inner product will be denoted by $\langle
\cdot,\cdot \rangle$.  $\Bo(X,Y)$ is the Banach algebra of bounded
linear operators from $X$ to $Y$, where $Y$ is another Banach space,
and $\Bo(X):=\Bo(X,X)$.

For a \Semi $T$ on $X$, $-A$ denotes its generator.
   %By $X_{1}$ we denote the domain $\DA$ of $A$ equipped with the graph norm.  
   The resolvent set of $A$ will be denoted by $\rho(A)$
   and $\sigma(A)$ refers to its spectrum.  For $\lambda\in\rho(A)$,
   $R(\lambda,A)=(\lambda\ide-A)^{-1}$.  $T$ is called an
   \textit{analytic} \Semi if it can be extended to a sector in the
   complex plane, see e.g.\ \cite[Def.~II.4.5]{EngelNagel}.

   For $\delta\in(0,\pi)$ define the sector
   $\Sigma_{\delta}=\left\{z\in\C:|z|>0,|arg(z)|<\delta\right\}$
   and set $\Sigma_{0}=(0,\infty)$.  A linear operator $A$ on $X$ is called
   \textit{sectorial of angle $\omega\in[0,\pi)$}, if
   $\sigma(A)\subset \overline{\Sigma_{\omega}}$ and for all
   $\delta\in(\omega,\pi)$
\begin{equation}\label{eq:sectorial}
	M(A,\delta):=\sup \left\{\|\lambda R(\lambda,A)\| : \lambda\in\C\setminus\overline{\Sigma_{\delta}}\right\} <\infty.
\end{equation}	
By $\Sect(\omega)$ we denote the set of sectorial operators on $X$ of angle
$\omega$. The minimal $\omega$ such that $A\in\Sect(\omega)$ is denoted by $\omega_{A}$. We recall that there is a one-to-one correspondence between
%densely-defined sectorial operators of angle strictly less than
%$\frac{\pi}{2}$ and generators of bounded analytic $C_{0}$-semigroups,
sectorial operators and generators of analytic semigroups,
namely, $A\in\Sect \left(\omega\right)$ with $\omega<\frac{\pi}{2}$ and $\overline{D(A)}=X$ if and only if 
 $-A$  generates a bounded analytic $C_{0}$-semigroup, see e.g.\ \cite[Thm.~II.4.6]{EngelNagel}. \newline
We  now briefly introduce the (holomorphic) functional calculus for sectorial operators. For a detailed treatment we refer the reader to the book of Haase, \cite{haasesectorial}.
Let $\Omega\subset\C$ be an open set and let $H(\Omega)$  be the analytic functions on $\Omega$. The Banach algebra of bounded
analytic functions on $\Omega$, equipped with
$\|f\|_{\infty,\Omega}:=\sup_{z\in\Omega}|f(z)|$, is denoted by
$\Hinf(\Omega)$. As we will mainly use sectors
$\Omega=\Sigma_{\delta}$, we abbreviate
$\|f\|_{\infty,\Sigma_{\delta}}$ by $\|f\|_{\infty,\delta}$ or write
$\|f\|_{\infty}$ if the set is clear from the context.  For
$\delta=\frac{\pi}{2}$ we will write
$\Hinf(\Cp)=\Hinf(\Sigma_{\delta})$. Furthermore, let us define
\begin{align*}
\Hinf_{(0)}(\Sigma_\delta )={}&\left\{f\in\Hinf (\Sigma_{\delta}): |f(z)|\leq C |z|^{-s}\text{ for some }C,s>0\right\},\\
	\HinfO (\Sigma_{\delta})={}&\left\{f\in\Hinf (\Sigma_{\delta}): |f(z)|\leq C \tfrac{|z|^{s}}{1+|z|^{2s}}\text{ for some }C,s>0\right\},
\end{align*}
which are the bounded analytic functions which decay polynomially at
$\infty$ (and $0$). %\newline 

Let $A$ be a sectorial operator of angle $\omega$. Then, the Riesz-Dunford integral
\begin{equation}\label{rieszdunford}
		f(A)=\frac{1}{2\pi i} \int_{\Gamma} f(z)R(z,A)\ dz,
\end{equation}
	is well-defined in $\Bo(X)$ in each of the following situations, with $\omega<\delta'<\delta<\pi$, 
	\begin{enumerate}
	\item $f\in\HinfO(\Sigma_{\delta})$ and $\Gamma=\partial\Sigma_{\delta'}$, where $\partial\Sigma_{\delta}$ denotes the boundary of $\Sigma_{\delta}$,
	\item $f\in\Hinf_{(0)}(\Sigma_{\delta})\cap H(B_{r}(0))$ for some $r>0$ and $\Gamma=\partial \left(B_{r'}(0)\cup\Sigma_{\delta'}\right)$ for $r'\in(0,r)$ ,
	\item $f\in\Hinf_{(0)}(\Sigma_{\delta})$, $0\in\rho(A)$ and $\Gamma=\partial\left(\{z:\Re z>r\}\cap\Sigma_{\delta'}\right)$ for $r>0$ sufficiently small,
	\end{enumerate}
	where $B_{r}(0)=\left\{z\in\C:|z|<r\right\}$. 
	The above paths $\Gamma$ are orientated positively and by Cauchy's theorem 
	it follows that the definitions are consistent and independent of the choice of $\delta'$ and $r'$. \newline
	The mapping $f\mapsto f(A)$ is an algebra homomorphism from $\HinfO(\Sigma_{\delta})$ to $\Bo(X)$. 
	It is straight-forward to extend it to a homomorphism $\Phi$ from
	$\mathcal{E}=\HinfO(\Sigma_{\delta})\oplus\langle1\rangle\oplus\langle\tfrac{1}{1+z}\rangle$ to $\Bo(X)$. 
	The tuple $(\mathcal{E},\mathcal{H}(\Sigma_{\delta}),\Phi)$ is called a \textit{primary calculus} which, by a \textit{regularization argument}, can be extended to more general 
	$f\in H(\Sigma_{\delta})$. This  algebraic procedure yields an, in general unbounded, calculus of closed operators. 
	The regularization argument can be sketched as follows. The set of \textit{regularizers} is defined as
	\begin{equation*}
	 {\rm Reg}_{A}=\left\{e\in\HinfO(\Sigma_{\delta}):e(A)\text{ is injective}\right\}
	\end{equation*}
	and the functions that can be \textit{regularized} by elements in ${\rm Reg}_{A}$ are
	\begin{equation*}
		\mathcal{M}_{A}=\left\{f\in H(\Sigma_{\delta}):\exists e\in {\rm Reg}\text{ with } (ef)\in\HinfO(\Sigma_{\delta})\right\}.
	\end{equation*}
	Then, for any $f\in\mathcal{M}_{A}$, we can define $f(A)=e(A)^{-1}(ef)(A)$ which turns out to be independent of the choice of $e$. If $A$ is injective, it holds that $\Hinf(\Sigma_{\delta})\subset \mathcal{M}_{A}$. 
	One can show that the extension procedure is in conformity with the Riesz-Dunford integral definition in items 2 and 3 above. Clearly, for invertible $A$ one can do the analogous construction with a primary calculus on $\Hinf_{(0)}(\Sigma_\delta)$, which extends the previous calculus.
	For  detailed and more general axiomatic treatment of the construction of the calculus 
	we refer to Chapter 1 and 2 in \cite{haasesectorial}. \newline
	Let  $(\mathcal{F},\|\cdot\|_{\mathcal{F}})$ be a  Banach algebra such that $\mathcal{F}$ is a subalgebra of $\Hinf(\Sigma_{\delta})$ and that $f(A)$ is defined by the above calculus for all $f\in\mathcal{F}$.
	Following Haase \cite[Chapter 5.3]{haasesectorial}, we say that the $\mathcal{F}$-calculus is \textit{bounded} if $f(A)$ is bounded for all $f\in\mathcal{F}$ 
	and 
	\begin{equation}\label{ineq:bddcalc}
		\exists C>0:\quad\|f(A)\|\leq C\|f\|_{\mathcal{F}},\quad \forall f\in \mathcal{F}.
	\end{equation}
	%The infimum over all possible $C$ is called the bound of the calculus. 
	For $\mathcal{F}$ closed with $\|\cdot\|_{\mathcal{F}}=\|\cdot\|_{\infty,\delta}$ and $A$ injective, (\ref{ineq:bddcalc}) follows already if $f(A)$ is bounded for all $f\in\mathcal{F}$, by the Convergence Lemma, \cite[Prop.~5.1.4]{haasesectorial} and the Closed Graph Theorem.
	\newline
	By $e_{\veps}$ we denote the function $z\mapsto e^{-\veps z}$ 
	which lies in $\Hinf_{(0)}(\Sigma_{\delta})$ for $\delta<\frac{\pi}{2}$ and $\veps>0$.
	
In the following the \textit{exponential integral} function
\begin{equation}\label{eq:expint}
			{\rm Ei}(x)=\int_{1}^{\infty}\frac{e^{-xt}}{t}\ dt,\quad x>0,
\end{equation}
will be used several times. It is clear that ${\rm Ei} (x)$ is
decreasing. The asymptotic behavior of ${\rm Ei}(x)$ is reflected in
the estimates
\begin{equation}\label{eq:expintEST}
			\tfrac{1}{2}e^{-x}\log\left(1+\tfrac{2}{x}\right) < {\rm Ei}(x)< e^{-x}\log\left(1+\tfrac{1}{x}\right), \quad x>0,
\end{equation}
which go back to Gautschi \cite{Gautschi59} and can also be found in
\cite[5.1.20]{AbramowitzExpInt}.  This implies that
\begin{equation}\label{eq:EiEST2}
	\tfrac{1}{2e}|\log(x)|<{\rm Ei}(x)<|\log (x)|, \quad x\in\left(0,\tfrac{1}{2}\right).
\end{equation}
%Thus, by (\ref{eq:expintEST}), ${\rm Ei}(x)\sim |\log x|$ for $x<\tfrac{1}{2}$.

   We write $E(z)\sim F(z)$, if there exist absolute constants $K_{1}, K_{2}>0$ with $K_{1}\leq E(z)\leq K_{2} F(z)$ for all considered $z$. For example, ${\rm Ei}(x)\sim |\log x|$ for $x<\tfrac{1}{2}$ by \eqref{eq:EiEST2}.
   %By $E(z)\sim F(z)$, we mean that $F(z)\lesssim
   %E(z)$ and $E(z)\lesssim F(z)$. 
	
%%%%%%%%%%%%%%%%%%%%%%%%%%%

\section{Main results}\label{sec:mainresults}
Unless stated explicitly, $X$ will always denote a general Banach space.
\subsection{Sectorial operators and functions holomorphic at $0$}

The following example shows that the assumption $0\in\rho(A)$ cannot be neglected, if we want to study estimates of the form \eqref{eq:int1} for $f\in \Hinf(\Cp)$.
\begin{example}\label{ex:invertibility}%
  Let $-B$ be the generator of the bounded analytic semigroup $S$ with
  $0\in\rho(B)$.  Assume that the $\Hinf(\Cp)$-calculus is not
  bounded, thus, there exists $f\in\Hinf(\Cp)$ such that $f(B)$ is
  unbounded.  Such examples exist even on Hilbert spaces, see e.g.\
  \cite{BaillonClement91} or Section \ref{sec:diagop}. Then, $A=B^{-1}$ is bounded, sectorial of
  the same angle as $B$, see \cite{haasesectorial}, and has dense
  range. Thus $g(A)$ is defined by the $\Hinf$-calculus for sectorial
  operators for $g\in\Hinf$ in some sector.  Furthermore, by the
  \textit{composition rule}, see \cite[Prop.~2.4.1]{haasesectorial}, we
  have that for $h=(z\mapsto z^{-1})$,
\begin{equation*}
			 (f\circ h)(A)=f(B),
\end{equation*}
		where $(f\circ h)\in\Hinf(\Cp)$.
		Since $A$ is bounded, $A$ even generates a group $T(\veps)=e_{\veps}(A)$. Hence, $\left((f\circ h)\cdot e_{\veps}\right)(A)=f(B)T(\veps)$
		cannot be bounded for any $\veps>0$.
\end{example} 
	
The reason why we cannot expect $(fe_{\veps})(A)$ to be a
bounded operator if $0\notin\rho(A)$ is that the integrand in
(\ref{rieszdunford}) may have a singularity at $0$. However, instead
of making the resolvent exist at $0$, we can pass over to a smaller set of functions in $\Hinf$.

\begin{proposition}\label{prop3}
	Let $A$ be a densely defined, sectorial operator of angle $\omega<\frac{\pi}{2}$ on the Banach space $X$ with dense range. Let $\phi\in(\omega,\pi)$ and $\mathcal{F}\subset\Hinf(\Sigma_{\phi})$ such that 
	\begin{enumerate}[label=(\roman*)]
		\item\label{ass1} $D(A)\subset D(f(A))$ for all $f\in\mathcal{F}$, and
		\item\label{ass2} $\forall$ $f\in \Hinf(\Sigma_{\phi})$ there exists  $\{f_{n}\}_{n\in\N}\subset\mathcal{F}$ such that $f_{n}\to f$ pointwise and $\sup_{n}\|f_{n}\|_{\infty,\phi}<\infty$.
		\end{enumerate}
		Then, the $\Hinf(\Sigma_{\phi})$-calculus is bounded if and only if 
		\begin{align}\label{eq23}
		\exists C>0\ \forall f\in\mathcal{F}:\quad \limsup\nolimits_{\veps\to0^{+}}\|(f e_{\veps})(A)\|< C \|f\|_{\infty,\phi}.
		\end{align}
		If $\mathcal{F}=\Hinf(\Sigma_{\phi})$, then $\|f\|_{\infty,\phi}$ in \eqref{eq23} can be replaced by any constant $C_{f}>0$.
\end{proposition}
\begin{proof}
		Note that $A$ is injective as it is a sectorial operator with dense range, see \cite[Prop.~2.1.1]{haasesectorial}.  Thus, $f(A)$ is defined as a closed operator for every $f\in\Hinf(\Sigma_{\phi})$.
		%For $g$ holomorphic at $0$, $(\frac{g(z)}{1+z})(A)$ is defined by (\ref{rieszdunford}), and hence bounded. 
		%Thus,  $D(A)\subset D(g(A))$. 
		 Since $e_{\veps}$ is holomorphic at $0$ and in $\Hinf_{(0)}(\Sigma_{\frac{\pi}{4}})$, 
  we have that $e_{\veps}(A)$ is bounded. Furthermore, $R(e_{\veps}(A))\subset D(A)$ for $\veps>0$, $\sup_{\veps>0}\|e_{\veps}(A)\|<\infty$ and $\lim_{\veps\to0^{+}}e_{\veps}(A)x=x$ for all $x\in\overline{D(A)}=X$, see \cite[Prop.~3.4.1]{haasesectorial}. Hence, by assumption \ref{ass1}, $f(A)e_{\veps}(A)=(fe_{\veps})(A)\in\Bo(X)$ for all $f\in\mathcal{F}$. \newline
 % $(fe_{\veps})\in\Hinf_{(0)}(\Sigma_{\delta})$ for some
 % $\delta<\frac{\pi}{2}$. Hence $e_{\veps}(A)$ and $(fe_{\veps})(A)$ are
%  bounded operators. 
   If the calculus is bounded, $\|f(A)\|\leq \tilde{C}\|f\|_{\infty}$ for some $\tilde{C}>0$ and all $f\in H^{\infty}(\Sigma_{\phi})$.  Thus,
\begin{equation*}
		\|(fe_{\veps})(A)\|=\|f(A)e_{\veps}(A)\| \leq \tilde{C}~ \|e_{\veps}(A)\|\ \|f\|_{\infty,\phi}\leq C \|f\|_{\infty,\phi},\quad f\in\mathcal{F},
\end{equation*}
where $\tilde{C}$ does not depend on $\veps$.  Therefore,
(\ref{eq23}) holds. 
Conversely, let (\ref{eq23}) be satisfied.  
For $f\in\mathcal{F}$ and $x\in\DA$, we have that
\begin{equation*}
	\|f(A)x\| \leq \|f(A)x-e_{\veps}(A)f(A)x\| + \|e_{\veps}(A)f(A)x \|.
\end{equation*}
For $\veps\to0^{+}$, the first term on the right-hand-side tends to
zero by the properties of $(e^{-\veps\cdot})(A)$, see above.
Since $e_{\veps}(A)f(A)x=(e_{\veps}f)(A)x$ for $x\in D(A)$, see \cite[Thm.~
1.3.2.c)]{haasesectorial}, the second term can be estimated by  the
assumption of (\ref{eq23}). As $D(A)$ is dense, we get that $f(A)$ is bounded and
\begin{equation}\label{prop2:eq2}
			\|f(A)\|\leq \limsup_{\veps\to0^{+}} \| (fe_{\veps})(A)\| \leq C \|f\|_{\infty,\phi}, \quad f\in\mathcal{F}.
		\end{equation}
		By assumption \ref{ass2}, \eqref{prop2:eq2} and the Convergence Lemma \cite[Prop.~5.1.4b)]{haasesectorial} (here, we use that $D(A)$ and $R(A)$ are dense), we conclude that $\|f(A)\|\leq C \|f\|_{\infty,\phi}$ for all $f\in \Hinf(\Sigma_{\phi})$.\newline
		 If $\mathcal{F}=\Hinf(\Sigma_{\phi})$ and if we replace $\|f\|_{\infty,\phi}$ by some constant $C_{f}>0$ in \eqref{eq23}, then, in \eqref{prop2:eq2}, we derive that $\|f(A)\|\leq C_{f}$ for $f\in\Hinf(\Sigma_{\phi})$, which implies that the $\Hinf(\Sigma_{\phi})$-calculus is bounded.
		\end{proof}
		Regarding Proposition \ref{prop3}, in this paper we will study the following situations.
			\begin{itemize}
			\item $\mathcal{F}=\Hinf(\Sigma_{\phi})$ und $A$ invertible (then, $(z\mapsto \frac{z}{(1+z)^{2}})$ is a regularizer for any $f$),
			\item $\mathcal{F}=\{f\in\Hinf(\Sigma_{\phi}):f\text{ holomorphic at }0\}$ (then, for $f\in\mathcal{F}$, $(\frac{f(z)}{1+z})(A)$ is defined by (\ref{rieszdunford})).
		\end{itemize}
		It is not hard to see that in the above cases, \ref{ass1} and \ref{ass2} from Proposition \ref{prop3} are fulfilled. Hence, Proposition \ref{prop3} implies Proposition \ref{prop1}.
		
%\begin{proof}
%		The proof is essentially the same as for Proposition \ref{prop1} with the following adaptions: 
%		Note that $A$ is injective as it is a sectorial operator with dense range, see \cite[Proposition 2.1.1]{haasesectorial}. Thus, the calculus is defined for $\Hinf(\Sigma_{\phi})$.
%		For $g$ holomorphic at $0$, $(\frac{g(z)}{1+z})(A)$ is defined by (\ref{rieszdunford}), and hence bounded. 
%		Thus,  $D(A)\subset D(g(A))$. 
%		Because $D(A)$ is dense, it follows analogously to the proof of Proposition \ref{prop1} that 
%		\begin{equation}\label{prop2:eq2}
%			\|g(A)\|\leq \limsup_{\veps\to0^{+}} \| (ge_{\veps})(A)\| \leq C \|g\|_{\infty,\phi},
%		\end{equation}
%		where the last inequality holds if (\ref{prop2:bddcalceq1}) holds. For arbitrary $f\in\Hinf(\Sigma_{\phi})$ take a sequence $g_{n}\in\Hinf(\Sigma_{\phi})$ which are holomorphic at $0$ and converge to $f$ pointwise in $\Sigma_{\phi}$ with $\sup_{n}\|g_{n}\|_{\infty,\phi}<\infty$. Applying the Convergence Lemma \cite[Proposition 5.1.4b)]{haasesectorial}, using (\ref{prop2:eq2}) and the fact that $D(A)$ and  $R(A)$ are dense, yields that $f(A)$ is bounded. 
%		\end{proof}
%
In the next theorem we estimate $\|(fe_{\veps})(A)\|$. 
In Section 3, we show that this estimate is sharp.			
\begin{theorem}\label{thm1}
	Let $A\in\Sect(\omega)$, $0<\omega<\phi<\frac{\pi}{2}$ and $\veps,r_{0}>0$.
			 Further, let $f\in\Hinf(\Omega_{\phi,r_{0}})$ with $\Omega_{\phi,r_{0}}:=\Sigma_{\phi}\cup B_{r_{0}}(0)$. 
			 Then $(fe_{\veps})(A)$ is bounded and 
			 \begin{equation}\label{eq:loganalytic2}
		 	\|(fe_{\veps})(A)\| \leq M(A,\phi) \cdot b(\veps,r_{0},\phi)\cdot \|f\|_{\infty,\Omega_{\phi,r_{0}}},
			 \end{equation}
			 with
			 \begin{equation}\label{thm1:b}
		 	 b(\veps,r_{0},\phi)=\frac{1}{\pi}\cdot\left\{\begin{array}{ll}{\rm Ei}(\veps r_{0}\cos\phi)+e^{\veps r_{0}}(\pi-\phi),&2\veps r_{0}\leq1,\\{\rm Ei}\left(\tfrac{\cos\phi}{2}\right)+\sqrt{e}(\pi-\phi),&2\veps r_{0}>1.\end{array}\right.
\end{equation}
 Here, ${\rm Ei}(x)$ is the the exponential integral, see (\ref{eq:expint})--(\ref{eq:EiEST2}), therefore, 
			\begin{equation}\label{eq:loganalyticNICE}
				b(\veps,r_{0},\phi) \sim \left\{ \begin{array}{ll}|\log(\veps r_{0} \cos\phi)|,&  \veps r_{0}<\tfrac{1}{2},\\  |\log\tfrac{\cos\phi}{2}|,&\veps r_{0}\geq\tfrac{1}{2}.\end{array}\right.
			\end{equation}
\end{theorem}		 
\begin{proof}
	Since $fe_{\veps}\in\Hinf_{(0)}(\Sigma_{\phi})\cap \Hinf(\Omega_{\phi,r_{0}})$, $(fe_{\veps})(A)$ is a bounded operator defined by (\ref{rieszdunford}). Hence,
			\begin{equation}\label{thm1:eq1}
			\|(fe_{\veps})(A)\|=\frac{1}{2\pi} \left\|\int_{\Gamma_{r}} f(z)e^{-\veps z}R(z,A) \ dz\right\|\leq \frac{\|f\|_{\infty,\Omega_{\phi,r_{0}}}}{2\pi} \int_{\Gamma_{r}} \|e^{-\veps z}R(z,A)\|\ |dz|,
			\end{equation}
			where the integration path is chosen to be $\Gamma_{r}=\Gamma_{1,r}\cup\Gamma_{2,r}\cup\Gamma_{3,r}$ with			
			\begin{equation*}
			\Gamma_{1,r}=\left\{\tilde{r}e^{i\delta},\tilde{r}>r\right\},
			\Gamma_{2,r}=\left\{re^{i\psi},|\psi|\geq \delta\right\},
			\Gamma_{3,r}=\left\{\tilde{r}e^{-i\delta}, \tilde{r}>r\right\},\quad r\in(0,r_{0}), \delta\in(\omega,\phi),
			\end{equation*}
			orientated positively.
%			Since $f\in\Hinf(\Omega_{\phi,r_{0}})$, we can estimate
%			\begin{equation}
%			\|(fe_{\veps})(A)\|\leq\frac{\|f\|_{\infty,\Omega_{\phi,r_{0}}}}{2\pi} \int_{\Gamma_{r}} \|e^{-\veps z}R(z,A)\|\ |dz|.
%			\end{equation}
			The rest of the proof is similar to the argument that $\sup_{\veps>0}\|e_{\veps} (A)\|<\infty$ for sectorial operators with $\omega_{A}<\frac{\pi}{2}$, see e.g.\ \cite{EngelNagel,Pazy83,Vitse05}.
			Splitting up the integral, for $z\in\Gamma_{1,r}$,
			\begin{equation*}
			\|e^{-\veps z}R(z,A)\|\leq e^{-\veps \Re z}\cdot  \frac{M(A,\delta)}{|z|}=\frac{e^{-\veps|z|\cos\delta}}{|z|}M(A,\delta).
			\end{equation*}
			On $\Gamma_{3,r}$ the same estimate holds.
			For $z\in\Gamma_{2,r}$, 
			%\begin{equation*}
		$		\|e^{-\veps z}R(z,A)\|\leq e^{\veps r} \cdot \frac{M(A,\delta)}{r}.$
		%	\end{equation*}
			Therefore,
			\begin{align}
			\int_{\Gamma r}\|e^{-\veps z}R(z,A)\| |dz| \leq{}& M(A,\delta)\left(2\int_{r}^{\infty}\frac{e^{-\veps \tilde{r}\cos\delta}}{\tilde{r}}d\tilde{r} +\frac{e^{\veps r} }{r}\int_{\Gamma_{2,r}}  |dz|\right) \notag\\
			\leq{}&2M(A,\delta)({\rm Ei}(\veps r\cos\delta)+e^{\veps r}(\pi-\delta)).\label{thm1:eq2}
			\end{align}
			Next, for $n\in\N$, we choose $r$ as 
			\begin{equation*}		
	r=\left\{\begin{array}{ll}r_{n}=r_{0}(1-2^{-n}),& 2\veps r_{0}\leq1,\\\frac{1}{2\veps},&2\veps r_{0}>1.\end{array}\right.
	\end{equation*}
Clearly, $r$ lies within $(0,r_{0})$. Hence, by (\ref{thm1:eq1}) and
(\ref{thm1:eq2}),
\begin{equation*}
			\|(fe_{\veps})(A)\| \leq \frac{M(A,\delta)}{\pi} \left\{\begin{array}{ll}{\rm Ei}(\veps r_{n}\cos\delta)+e^{\veps r_{n}}(\pi-\delta),&2\veps r_{0}\leq1,\\{\rm Ei}\left(\tfrac{\cos\delta}{2}\right)+\sqrt{e}(\pi-\delta),&2\veps r_{0}>1\end{array}\right\}. 
		%	M(A,\delta)\: b(\veps,r_{n},\delta)\:\|f\|_{\infty,\Omega_{\phi,r_{0}}}.
			\end{equation*}
			Letting $n\to\infty$ and $\delta\to\phi^{-}$ shows the assertion.
\end{proof}
			%As $\Hinf(\Omega_{\frac{\pi}{2},r_{0}})$ is continuously embedded in $\Hinf(\Cp)$, and since  $\|f\|_{\infty,\Cp}=\|e_{\veps}f\|_{\infty,\Cp}$ we have the following direct consequence of Theorem \ref{thm1}.
			By $\|f\|_{\infty,\Omega_{\frac{\pi}{2},r}} \leq \|e_{\veps}f\|_{\infty,\Omega_{\frac{\pi}{2},r}}$ for $f\in\Hinf(\Omega_{\frac{\pi}{2},r})$, the following consequence of Theorem \ref{thm1} holds.
			\begin{corollary}\label{corH+}
			For $A\in\Sect(\omega)$, $\omega<\frac{\pi}{2}$, $\veps,r>0$, the 
			$e_{\veps}H^{\infty}(\Cp\cup B_{r}(0))$-calculus is bounded.% with $\Omega_{\frac{\pi}{2},r}=\Cp\cup B_{r}(0)$.

			\end{corollary}
		%	Note that $e_{\veps}H^{\infty}(\Omega_{\frac{\pi}{2}},r)$ is a closed ideal in $\Hinf(\Cp)$.
\subsection{The space $H^{\infty}[\veps,\sigma]$ and Vitse's result}
\label{sec:Vitse}%
In this subsection we show that the result in Theorem \ref{thm1}
generalizes Theorem 1.6 in \cite{Vitse05}.
\medskip

For $\veps,\sigma\in\R$ with $0\leq\veps<\sigma\leq\infty$, let
$H^{\infty}[\veps,\sigma]$ denote the space of functions which are in
$\Hinf(\Cp)$ and are the Laplace-Fourier transform of a distribution
supported in $[\veps,\sigma]$.  For $\sigma=\infty$, we get $H^{\infty}[\veps,\infty]=e^{-\veps z}\Hinf(\Cp)$. 
Recall that an entire function $g$ is
of \textit{(exponential) type $0<\sigma<\infty$} if for any $\epsilon>0$ there
exists $C_{\epsilon}>0$ such that $|g(z)|\leq
C_{\epsilon}e^{(\sigma+\epsilon)|z|}$ for all $z\in\C$.
\newline
For $\sigma<\infty$, the following Paley-Wiener-Schwartz type result
holds, see \cite[p.174]{HavinJoericke}.
\begin{equation}\label{eq:CharHepssigma}
	g\in H^{\infty}[\veps,\sigma]  \iff g\text{ is entire of exponential type }\sigma \text{ and } ge^{\veps\cdot}\in\Hinf(\Cp).
\end{equation}
For more details about $H^{\infty}[\veps,\sigma]$, we refer to \cite{Vitse05} and the references therein.\newline
The following is a consequence of the  Phragm\'en-Lindel\"of principle, see \cite[Thm.~6.2.4, p.82]{BoasEntireFunctions}.
\begin{lemma}\label{lemBoas}
Let $g:\C\rightarrow\C$ be entire of exponential type $\sigma$ such that $\|g\|_{\infty,i\R}<\infty$. Then, %for $x,y\in\R$, %=\sup_{y\in\R}|g(iy)|
\begin{equation*}
	|g(x+iy)| \leq e^{\sigma|y|}\ \|g\|_{\infty,i\R}, \qquad \forall x,y\in\R.
\end{equation*}
\end{lemma}
Using Lemma \ref{lemBoas}, Theorem \ref{thm1} yields an estimate in the $\Hinf(\Cp)$-norm.
\begin{theorem}\label{thm2}%
  Let $A\in\Sect(\omega)$, $\omega<\frac{\pi}{2}$ and
  $0<\veps<\sigma<\infty$.  With $b$ from \eqref{thm1:b}, the following holds.
\begin{equation}
\forall g\in H^{\infty}[\veps,\sigma]:\quad	\|g(A)\| \leq   \|g\|_{\infty,\Cp} \cdot \inf_{\phi\in(\omega,\frac{\pi}{2}),k\geq1}M(A,\phi)\ b\left(\veps,\tfrac{1}{k\sigma},\phi\right)  e^{\frac{\sigma-\veps}{k\sigma}}.
\end{equation}
%where $b(\veps,r,\phi)$ is defined in (\ref{thm1:b}).
\end{theorem}
\begin{proof}
  Let $f(z)=e^{\veps z} g(z)$. By (\ref{eq:CharHepssigma}), $f$ lies
  in $ \Hinf(\Cp)$ and is entire of type $\sigma-\veps$.
  Let $k\geq1$. Since $f$ is entire and bounded on $\Cp$, we can
  apply Theorem \ref{thm1} with $r_{0}=\frac{1}{k\sigma}$.  
	Thus, for $\phi\in(\omega,\tfrac{\pi}{2})$,
\begin{equation*}
	\|g(A)\| = \|(fe_{\veps})(A)\| \leq \inf_{\phi\in(\omega,\frac{\pi}{2})}M(A,\phi) \cdot b\left(\veps,\tfrac{1}{k\sigma},\phi\right)\cdot \|f\|_{\infty,\Omega_{\phi,\frac{1}{k\sigma}}},
\end{equation*}
where $\Omega_{\phi,\frac{1}{k\sigma}}=\Sigma_{\phi}\cup
B_{\frac{1}{k\sigma}}(0)$.
  Clearly, $\|f\|_{\infty,\Omega_{\phi,\frac{1}{k\sigma}}} \leq
\|f\|_{\infty,\Cp\cup B_{\frac{1}{k\sigma}}(0)}$.  Moreover, as $f$ is
entire of exponential type $\sigma-\veps$ and
$\sup_{y\in\R}|f(iy)|=\|f\|_{\infty,\Cp}$, we can apply Lemma \ref{lemBoas} to conclude that
				 \begin{equation*}
		 \|f\|_{\infty,\Omega_{\phi,\frac{1}{k\sigma}}}\leq  e^{\frac{\sigma-\veps}{k\sigma}} \|f\|_{\infty,\Cp}.
				 \end{equation*}
Since $\|g\|_{\infty,\Cp}=\|f\|_{\infty,\Cp}$, the assertion follows.
\end{proof}
	
Now we write Theorem \ref{thm2} in the terminology used in
\cite{Vitse05}.  %This will reveal that the dependence on  $M(A,\phi)$ of our
%approach improves the corresponding estimate in \cite{Vitse05}. 
There, for $\theta\in(0,\pi]$, a densely defined closed
operator is called $\theta$-\textit{sectorial}, if $\sigma(A)$ is
contained in $\Sigma_{\theta}\cup\left\{0\right\}$ (note that in our
definition of $\Sect(\theta)$, $\sigma(A)$ is contained in
$\overline{\Sigma}_{\theta}$) and
\begin{equation*}
			 \tilde{M}(A,\theta)=\sup \left\{\ \|zR(z,A)\| : z\in\C\setminus(\Sigma_{\theta}\cup\left\{0\right\})\ \right\} <\infty.
\end{equation*} 
By $S(\theta)$ let us denote the $\theta$-sectorial operators on $X$.
As pointed out in \cite[Sec.~1.1]{Vitse05}, 
$S(\theta)\subset \Sect(\theta)\subset S(\theta+\epsilon)$ for all
$\epsilon>0$ and $S(\theta)=\bigcup_{0<\theta'<\theta}
\Sect(\theta')$.  Moreover, for $A\in S(\frac{\pi}{2})$ there exists $\theta<\frac{\pi}{2}$ such that
$A\in S(\theta)$, see Lemma \ref{lemVitse} below. 
 Hence, $A\in \Sect(\theta)$ for some
$\theta<\frac{\pi}{2}$ if and only if $A\in S(\frac{\pi}{2})$.  Furthermore, for $A\in
S(\theta)$ we have by continuity that
			 \begin{equation}%\label{rem1:eq1}
		 \tilde{M}(A,\theta)=\sup_{z\in\C\setminus(\Sigma_{\theta}\cup\left\{0\right\})} \|zR(z,A)\|=\sup_{z\in\C\setminus\overline{\Sigma_{\theta}}} \|zR(z,A)\|=M(A,\theta).
\end{equation}
The following is a well-known consequence of a Neumann series argument, e.g.\ \cite[Lem.\ 1.1]{Vitse05}.
\begin{lemma}\label{lemVitse}
Let $A\in S(\tfrac{\pi}{2})$ and $M=\tilde{M}(A,\tfrac{\pi}{2})$. Then, $A\in S(\theta)$ for 
		\begin{equation*}%\label{eq3:cor1}
			\theta=\arccos\frac{1}{2M}\in\theta\in\left(\tfrac{\pi}{3},\tfrac{\pi}{2}\right)\quad \text{and}\quad \tilde{M}(A,\theta)=M(A,\theta)\leq2 M.
\end{equation*}
% Note that  $S(\theta)=\bigcup_{0<\theta'<\theta}\Sect(\theta')$.
% Since $M\geq1$,  $\theta\in(\frac{\pi}{3},\frac{\pi}{2})$.
% (\cite[Prop.2.1.1]{haasesectorial}),
\end{lemma}

\begin{theorem}\label{cor1}%
  Let $A\in S(\tfrac{\pi}{2})$ and $M=\tilde{M}(A,\frac{\pi}{2})$.
  %, which is equivalent to $A\in\Sect(\omega)$ with $\omega<\frac{\pi}{2}$ (see above).
  Then, 
  %\begin{enumerate}[label={(\roman*)}]
	%\item \label{cor1:it1} For all $t\geq0$,
	%	 			\begin{equation}\label{eq:SGbound}
	%				\|e^{-tA}\| \leq   \frac{2M}{\pi}(\log(M)+5). 
	%				\end{equation}
	 %\item\label{cor1:it3} 
	 for $0<\veps<\sigma<\infty$ and $g\in H^{\infty}[\veps,\sigma]$,
	\begin{equation}\label{cor1:eq}
	\|g(A)\| \leq \left(C_{1} + C_{2} \log\left(\frac{\sigma}{\veps}\right)\right) \|g\|_{\infty,\Cp}\leq C_{3}\log\left(\frac{\sigma e}{\veps}\right) \|g\|_{\infty,\Cp},
	\end{equation}
%\end{enumerate}
with $C_{1} = c_{1}M+c_{2}M\log(M)$, $C_{2} = c_{2}M$ and $C_{3} =
c_{1}M+c_{2}M\log(M)$ and
	\begin{align*}
				 c_{1}=\tfrac{2e^{\frac{1}{5}}}{\pi}\left(\log(10)+\tfrac{2\pi}{3}\right)\approx 3.42,\ c_{2}=\tfrac{2e^{\frac{1}{5}}}{\pi}\approx0.78. 
\end{align*}
\end{theorem}	 
\begin{proof}
  \mbox{} Let $\theta$ be defined as in Lemma \ref{lemVitse}, hence, $\theta\in(\tfrac{\pi}{3},\tfrac{\pi}{2})$, $\cos\theta=\frac{1}{2M}$, and $M(A,\theta)\leq2 M$.
		Using Theorem \ref{thm2}, we get
		\begin{equation*}
		\|g(A)\| \leq \tfrac{2M}{\pi} \cdot \|g\|_{\infty,\Cp}\cdot \inf_{k\geq1}b(\veps,\tfrac{1}{k\sigma},\theta)e^{\frac{\sigma-\veps}{k\sigma}}.
		\end{equation*}
		It remains to estimate the infimum. For $k\geq2$, $\frac{\veps}{2Mk\sigma}<\frac{\veps}{k\sigma}<\frac{1}{2}$ and thus, by (\ref{thm1:b}) and (\ref{eq:EiEST2}), we get for  $b=b(\veps,\frac{1}{k\sigma},\theta)$ 
		  that
		\begin{align*}\label{eq4:cor1}
		b\cdot e^{\frac{\sigma-\veps}{k\sigma}}=\left[{\rm Ei}\left(\frac{\veps}{2Mk\sigma}\right)e^{\frac{\sigma-\veps}{k\sigma}}+e^{\frac{1}{k}}\frac{2\pi}{3}\right]\leq{}& \left[ \log\left(\frac{2Mk\sigma}{\veps}\right)e^{\frac{\sigma-\veps}{k\sigma}}+e^{\frac{1}{k}}\frac{2\pi}{3}\right].
		\end{align*} 
	%To  prove \ref{cor1:it1}, let $\veps=t$ and set $g=e_{t}\in H^{\infty}[t,\sigma]$.  Then, let $\sigma\to \veps^{+}=t^{+}$ in (\ref{eq4:cor1}) and choose $k=5$ (alternatively, apply Theorem
%\ref{thm1} with $f(z)=1$, $\veps=t$ and $r_{0}=\frac{1}{5\veps}$).\newline
	%To show \ref{cor1:it3}, 
	Using $e^\frac{\sigma-\veps}{k\sigma}<e^{\frac{1}{k}}$,  the right-hand-side % of (\ref{eq4:cor1}) 
	can be further estimated,
		\begin{equation*}
		b\cdot e^{\frac{\sigma-\veps}{k\sigma}}\leq \left[ \log(M)+\log\left(\frac{\sigma}{\veps}\right)+\log(2k)+\frac{2\pi}{3}\right]\cdot e^{\frac{1}{k}}.
		\end{equation*}
		Setting $k=5$, we get the result. 
		 \end{proof}
\begin{remark}
\begin{enumerate}
\item In \cite[Lem.~1.2 and Thm.~1.6]{Vitse05}, Vitse derives
  similar estimates as in Theorem \ref{cor1}.  However, she
  uses the Hille-Phillips calculus and considers elements of
  $H^{\infty}[\veps,\sigma]$ that are Laplace transforms of
  $L^{1}(\veps,\sigma)$-functions first.  The approach moreover relies on
  estimates of derivatives of the (analytic) semigroup.  This results
  in a similar estimate as in (\ref{cor1:eq}),
  but with the following constants
%\begin{equation*}
$		 \tilde{C}_{1}=\frac{30}{\pi}M^{2},\quad \tilde{C}_{2}=\frac{16}{\pi}M^{3}, \quad \tilde{C}_{3}=\frac{30}{\pi}M^{3}.$
%\end{equation*}
Thus, by our results, the $M$-dependence gets improved from $M^{3}$ to $M(1+\log M)$.
%Moreover, a more careful study even shows that $C_{i}\leq
%\tilde{C}_{i}$, $i\in\left\{1,2,3\right\}$, for every $M\geq1$.
				\item
		 	We point out that Vitse uses an estimate for the semigroup, \cite[Lem.~1.2]{Vitse05} 
			%(which is slightly improved by (\ref{eq:SGbound})), 
			to obtain an estimate for $H^{\infty}[\veps,\sigma]$ functions, 
			whereas our estimates all follow directly from Theorem \ref{thm2}.
			In other words, (the estimate for) the dependence on $M$ is the same for any 
			$H^{\infty}[\veps,\sigma]$ function, including $e_{\veps}$. In particular, \eqref{cor1:eq} implies that  
		%	\begin{equation*}%\label{eq:SGbound}
		$		\|e^{-\veps A}\| \leq   \frac{2M}{\pi}(\log(M)+6)$ for $\veps>0.$
		%	\end{equation*}	
			\item Possibly, $c_{1}$ and $c_{2}$ in Theorem \ref{cor1} can be further improved by optimizing $k$ in the proof. 
		\end{enumerate}
		 \end{remark}
		 \subsection{Invertible $A$ - exponentially stable semigroups}
		 In the view of Proposition \ref{prop3}, we now consider the case with $\mathcal{F}=\Hinf(\Sigma_{\phi})$ and invertible $A$.
%                 Theorems \ref{thm1} and \ref{thm2} deal with the
%                 situation of \textit{bounded} analytic semigroups and
%                 functions $f$ which are holomorphic at $0$.  As might
%                 be expected, a similar result holds for functions $f$
%                 not necessarily holomorphic at $0$, but with a
%                 sectorial operator $A$ having $0\in\rho(A)$. 
		
		\begin{theorem}\label{thm:expstab}
			Let $A\in\Sect(\omega)$, $\omega<\phi<\pi/2$, and $0\in\rho(A)$. 
			Then, for $\veps>0$, $f\in\Hinf(\Sigma_{\phi})$ the operator $(fe_{\veps})(A)$ is bounded and for all $\kappa\in(0,1)$,
			\begin{equation}\label{eq11:expstab}
			\|(fe_{\veps})(A)\| \leq \frac{M(A,\phi)}{\pi} \cdot b_{\kappa}\left(\veps,\tfrac{1}{\|A^{-1}\|},\phi\right)\cdot \|f\|_{\infty,{\phi}}.	
			\end{equation}
			 Here, 
			\begin{equation}
			b_{\kappa}(\veps,R,\phi)={\rm Ei} \left(\veps \kappa R\cos\phi\right)+\frac{\kappa}{1-\kappa}e^{-\veps \kappa R\cos\phi}, 
			\end{equation}
			Hence, $b_{\kappa}(\veps,R,\phi)\sim C_{\kappa} |\log(\veps R\cos\phi)|$ for $\veps R\leq\frac{1}{2\kappa}$ and
			$b_{\kappa}(\veps,R,\phi)\sim  C_{\kappa} e^{-(\veps \kappa R\cos\phi)}$ for $\veps R\kappa>\frac{1}{2}$.
			%$\|(fe_{\veps})(A)\|$ goes to zero
                        %exponentially as $\veps\to\infty$ by the
                       % properties of ${\rm Ei}$, see
                     %   (\ref{eq:expintEST}) and (\ref{eq:EiEST2}).
		\end{theorem}
		
		\begin{proof}
			Since $0\in\rho(A)$ and $fe_{\veps}\in \Hinf_{(0)}(\Sigma_{\phi})$, 
			$(fe_{\veps})(A)$ is well-defined by (\ref{rieszdunford}), 
			\begin{equation*}
		(fe_{\veps})(A)=\frac{1}{2\pi i}\int_{\partial \Sigma_{\theta}} f(z)e^{-\veps z} R(z,A) \ dz,\quad  \theta\in(\omega,\phi).
			\end{equation*}
			% and 
			%where $\partial \Sigma_{\theta}$ denotes the boundary
			% (orientated positively) of $\Sigma_{\theta}$. 
			Since $0\in\rho(A)$, we have that the ball $B_{\frac{1}{\|A^{-1}\|}}(0)$ lies in $\rho(A)$. For $\kappa\in(0,1)$ set $r=\frac{\kappa}{\|A^{-1}\|}$. By Cauchy's theorem, we can replace the integration path $\partial\Sigma_{\theta}$ by $\Gamma=\Gamma_{1}\cup\Gamma_{2}\cup\Gamma_{3}$ with

\begin{equation*}
			\Gamma_{1} =\left\{se^{i\theta}, s\geq r\right\},
			\Gamma_{2}=\left\{re^{i\theta}-it, t\in(0,2\Im(re^{i\theta}))\right\},
			\Gamma_{3}=\left\{-se^{-i\theta}, s\leq -r\right\}.
			\end{equation*}			 
			Thus,
				\begin{equation}\label{eq2:thmexpstab}
		 \|(fe_{\veps})(A)\| \leq \frac{\|f\|_{\infty,\phi}}{2\pi} \int_{\Gamma} e^{-\veps\Re z} \|R(z,A)\|\ |dz|.
				\end{equation}
			By the resolvent identity, $\|R(z,A)\|\leq \frac{\|A^{-1}\|}{1-|z| \|A^{-1}\|}$, and thus, for $\kappa\in(0,1)$,
				\begin{equation*}
		\|R(z,A)\|\leq \frac{\|A^{-1}\|}{1-\kappa} \quad\text{for}\quad |z|\leq r=\frac{\kappa}{\|A^{-1}\|}.
				\end{equation*}
			This yields, since $\Gamma_{2}\subset B_{r}(0)$,
			\begin{align*}
		\int_{\Gamma} e^{-\veps\Re z} \|R(z,A)\|\ |dz|\leq {}&\tfrac{\|A^{-1}\|}{1-\kappa}\int_{\Gamma_{2}} e^{-\veps r\cos\theta} \ dt + 2M(A,\theta)\int_{r}^{\infty} \frac{e^{-\veps s \cos\theta}}{s}\ ds\\
		={}&\tfrac{2\|A^{-1}\|}{1-\kappa}\ r\sin\theta\ e^{-\veps r\cos\theta}+2M(A,\theta) {\rm Ei}(\veps r \cos\theta),\\
		\leq {}&2M(A,\theta)\left(\tfrac{\kappa}{1-\kappa}e^{-\veps r\cos\theta}+ {\rm Ei}(\veps r \cos\theta)\right),
		\end{align*}
		as $M(A,\theta)\geq1$, see e.g. \cite[Prop.~2.1.1]{haasesectorial}. Letting $\theta\to\phi^{-}$ yields the assertion.
		\end{proof}
		%\begin{remark}
		%If $A$ is sectorial and $R>0$, then clearly $RA$ is sectorial of the same  angle.
		%Since $f\mapsto f_{R}=f(R\cdot)$ is an isometric isomorphism on $H^{\infty}(\Sigma_{\phi})$, 
		%and $(fe_{\veps})(RA)=(f_{R}e_{\veps R})(A)$ by the composition rule of holomorphic functional calculus \cite[Theorem 2.4.2]{haasesectorial}, we see that it is sufficient to consider $\frac{1}{\|A^{-1}\|}=1$ in the proof of Theorem \ref{thm:expstab}.
		%\end{remark}
                Applying Theorem \ref{thm:expstab} to $f\equiv1$ shows
                that $\|e_{\veps}(A)\|$ decays exponentially for
                $\veps\to\infty$.  This behavior is natural as the
                condition that $0\in\rho(A)$ implies that the analytic
                semigroup is exponentially stable.  However, for
                $\veps\to0$, the theorem gives no bound for the norm.
                This can be derived by Theorem \ref{thm1} as we will
                see in the following result.
       \begin{corollary}\label{cor:expstab}
Let $A\in\Sect(\omega)$ and $\omega<\phi<\frac{\pi}{2}$. If $A$ is invertible,
then we define $R=\frac{1}{\|A^{-1}\|}$, otherwise we set $R$ to be zero.  
Then, for any $\kappa\in[0,1)$, there exists a $C>0$ such that
\begin{equation}\label{eq:expdecaySG}
		\|e_{\veps}(A)\| \leq C e^{-\veps \kappa R\cos\phi},\quad \veps>0,
		\end{equation}
		with $C\leq C_{\kappa}M(A,\phi){\rm Ei}(\cos\phi)$.
		\end{corollary}
		\begin{proof}
		Let $f\equiv1$. If $\veps\kappa R>1$, by (\ref{eq:expintEST}),
		\begin{equation*}
		{\rm Ei}(\veps\kappa R\cos\phi)<e^{-\veps\kappa R\cos\phi}\log\left(1+\tfrac{1}{\cos\phi}\right)<2e^{2}e^{-\veps\kappa R\cos\phi}{\rm Ei}(\cos\phi),
		\end{equation*} 
		where we used that ${\rm Ei}(2\cos\phi)<\rm Ei(\cos\phi)$ in the last inequality.
		Using this, Theorem \ref{thm:expstab} yields
		\begin{equation}\label{eq:proofexpdecay}
		 \|e_{\veps}(A)\| \leq \tilde{C}_{\kappa}M(A,\phi){\rm Ei}(\cos\phi)e^{-\veps\kappa R\cos\phi},  \quad \veps\kappa R>1,
		\end{equation}
		where $\tilde{C}_{\kappa}>0$ only depends on $\kappa$.\newline
		Now, let $\veps \kappa R\leq 1$.  
		We apply Theorem \ref{thm1} with $r_{0}=\frac{1}{\veps}$. 
		It implies that there exists an absolute constant $C_{2}$ such that $\|e_{\veps}(A)\|\leq C_{2} M(A,\phi) {\rm Ei}(\cos\phi)$. 
		Together with (\ref{eq:proofexpdecay}) the assertion follows.
		\end{proof}
		
                Let us point out that the corollary is interesting in
                terms of the dependence on the constants $M(A,\phi)$,
                $\|A^{-1}\|$ and $\phi$, whereas the exponential decay
                is clear by $0\in\rho(A)$. 
                Further note that the use of the scaling variable
                $\kappa$ is not so artificial as it might seem: By
                $B_{\|A^{-1}\|^{-1}}(0)\subset\rho(A)$, we have
                that the growth bound $\omega_{0}$ of the semigroup
                satisfies
                $\omega_{0}\leq-\frac{\cos\phi}{\|A^{-1}\|}$.  It is
                well-known that, even in the case of a
                \textit{spectrum-determined} growth bound, as we have
                it for analytic semigroups, this rate need not be
                attained, see e.g.\ \cite[Ex.~I.5.7]{EngelNagel}. The $\kappa$ encodes that we can
                achieve any exponential decay of rate
                $\tilde{\omega}\in(-\frac{\cos\phi}{\|A^{-1}\|},0]$.
        
               A version of the following result can already be found in \cite[Thm.~II.6.13]{Pazy83}, but the constant dependence is unclear there. 
               
 \begin{lemma}\label{le:fracT(t)growth}
Let $A\in\Sect(\omega)$ with $\omega<\phi<\frac{\pi}{2}$ and $\alpha\in(0,1]$. Set $R=\frac{1}{\|A^{-1}\|}\geq0$ (see Cor.~\ref{cor:expstab}). Then, for every $\kappa\in[0,1)$ there exists $C=C_{\alpha,\kappa} M(A,\phi) (\cos\phi)^{-\alpha}>0$ such that 
\begin{equation}\label{eq:fracgrowth2}
	  \|A^{\alpha}T(t)\| \leq C t^{-\alpha}e^{-t\kappa R\cos\phi } \qquad\forall t>0.
	\end{equation}
	%Note that by the assumptions, the growth bound $\omega_{0}$ of $T$ satisfies $\omega_{0}\leq -R\cos\omega$.
\end{lemma}
\begin{proof}It is easy to see that $A^{\alpha}T(t)$ is defined by \eqref{rieszdunford} with the same integration path $\Gamma$ as in Theorem \ref{thm:expstab} (with $\Gamma=\partial\Sigma_{\phi}$ for $R=0$). 
The estimate follows similarly as in Theorem \ref{thm:expstab}.
\end{proof}

\section{Sharpness of the result}
\subsection{Diagonal operators on Schauder bases (Schauder multiplier)}\label{sec:diagop}

 A typical construction of an unbounded calculus goes back to Baillon and Clement \cite{BaillonClement91} and has been used extensively since then, see \cite{Fackler14Reg} and the references therein. The situation is as follows.
Let $\left\{{\Phi}_{n}\right\}_{n\in\N}$ be a Schauder basis of the Banach space $X$. For the sequence $\mu=(\mu_{n})_{n\in\N}$ define the multiplication operator $\mathcal{M}_{\mu}$ by its action on the basis, i.e. $\mathcal{M}_{\mu}\Phi_{n}=\mu_{n}\Phi_{n}$, $n\in\N$, with maximal domain. The choice $\lambda_{n}=c^{n}$, $c>1$, yields a sectorial operator $A=\mathcal{M}_{\lambda}\in\Sect(0)$ with $0\in\rho(\mathcal{M}_{\lambda})$, and  for $f\in\Hinf(\Cp)$,
 	\begin{equation}\label{eq:multicalc}
 		\begin{array}{l}f(A)=f(\mathcal{M}_{\lambda})=\mathcal{M}_{f(\lambda)},\\ 
		D(\mathcal{M}_{f(\lambda)})=\left\{x=\sum_{n\in\N}x_{n}\Phi_{n}\in X: \sum_{n\in\N} f(\lambda_{n})x_{n}\Phi_{n} \text{ converges}\right\}.\end{array}
	 \end{equation}
	 	%See e.g.\ \cite[Chapter 9]{haasesectorial} and\newline
		Because of (\ref{eq:multicalc}), a way of constructing unbounded calculi consists of the following two steps:
\begin{enumerate}[label=(\Alph*)]
	\item\label{uCalc1} Find a sequence $\mu\in\ell^{\infty}(\N,\C)$ such that $\mathcal{M}_{\mu}\notin\Bo(X)$. 
	\item\label{uCalc2} Find $f\in\Hinf(\Cp)$ such that $f(\lambda_{n})=\mu_{n}$ for all $n\in\N$.
\end{enumerate}
Since $\left(\lambda_{n}\right)$ is interpolating, \cite{Garnett}, \ref{uCalc2} is always possible. Note that \ref{uCalc1} follows if we can 
\begin{equation}\label{eq:conditionality}
\text{ find }x\in X \text{ such that } x=\sum\nolimits_{n\in\N} x_{n}\Phi_{n} \text{ does not converge unconditionally}.
\end{equation}
In fact, then there exists a sequence $\mu_{n}\subset\left\{-1,1\right\}$ such that  $\sum_{n\in\N} \mu_{n}x_{n}\Phi_{n}$ does not converge. Thus, $x\notin D(\mathcal{M}_{\mu})$, and so $\mathcal{M}_{\mu}\notin\Bo(X)$. \newline
Conversely, this indicates that \textit{a bounded $\Hinf$-calculus implies a large amount of unconditionality}, \cite[p.124]{haasesectorial}, which can be made rigorous, see \cite[Sec.~5.6]{haasesectorial} and \cite{KunstmannWeis04}. 
For more information about unbounded $\Hinf$-calculi via diagonal operators, see \cite[Chapter 9]{haasesectorial} and  \cite{Fackler14Reg}.

Let $\left\{\Phi_{n}\right\}_{n\in I}$, $I\subset \N$, be a Schauder basis of a Banach space $X$. 
For finite $\sigma\subset I$, $P_\sigma$ denotes the projection onto the linear span of $\left\{\Phi_{n}\right\}_{n\in\sigma}$.
Let us introduce the following constants,
\begin{equation}\label{def:basisconstants}
	m_{\Phi}=\sup_{n\in I}\|P_{\left\{n\right\}}\|,\quad \kappa_{\Phi}=\sup_{k\leq\ell}\left\|P_{[k,\ell]\cap I}\right\|,\quad ub_{\Phi}=\sup_{\sigma\subset I,|\sigma|<\infty}\|P_{\sigma}\|.
\end{equation}
  $\kappa_\phi$ is called the \textit{basis constant} and $ub_{\Phi}$ the \textit{uniform basis constant} of $\left\{\Phi_{n}\right\}_{n\in I}$. Clearly,
 \begin{equation}
 	m_{\Phi} \leq \kappa_{\Phi} \leq ub_{\Phi}.
 \end{equation}

\pagebreak[2]
\begin{theorem}\label{thm:GenMult}
There exist $K_{0}, K_{1}>0$ such that the following holds.
	Let $\left\{\Phi_{n}\right\}_{n\in\N}$ be a Schauder basis on a Banach space $X$ with $\kappa_{\Phi}<\infty$
	and let $\lambda_{n}=c^n$, $n\in\N$ for $c>1$. 
	Then $A:=\mathcal{M}_\lambda\in\Sect(0)$ and
	\begin{enumerate}[label={(\roman*)}]
	\item\label{thm:GenMultit1} $M(A,\psi)\leq \kappa_{\Phi} M(\psi)$ for all $\psi\in(0,\pi]$, where $M(\psi)$ only depends on $\psi$.
	\item\label{thm:GenMultit2} $0\in\rho(A)$ and ${\rm dist}(\sigma(A),0)=c$.
%	\item\label{thm:GenMultit3} For $\veps>0$ and $N_{\veps}=\lfloor \frac{2{\rm Ei}(\veps)}{\log c}\rfloor$, there holds 
%		\begin{equation}\label{eq:MultGeneral}
%			\|(fe_{\veps})(A)\| \leq   \left( \pi\cdot ub_{\left\{\Phi_{n}\right\}_{n=1}^{N_{\veps}}}+m_{\Phi}e^{-k(\veps)}\left(\tfrac{K_{1}}{\log c}+1\right)\right)\|f\|_{\infty,\psi},
%		\end{equation}
%		for all $f\in\Hinf(\Sigma_{\psi})$, $\psi\in(0,\frac{\pi}{2})$ and $\veps>0$ and 
%			\begin{equation}\label{thm:MultGeneq2}
%		k(\veps)=\left\{\begin{array}{ll}K_{0}&\veps\leq\veps_{c},\\ \max\left\{K_{0},c\veps\right\}&\veps>\veps_{c},\end{array}\right. 
%	\end{equation}
%	with absolute constants $K_{0}$, $K_{1}>0$ and $\veps_{c}$ such that $2{\rm Ei}(\veps_{c})<\log c$.
	\item\label{thm:GenMultit3} % Choose $\veps_{c}>0$ such that $2{\rm Ei}(\veps_{c})<\log(c)$. 

	For $\veps>0$, let  
	\begin{equation}\label{thm:MultGeneq2}\def\arraystretch{1.3}
	N_{\veps}=\left\lfloor \tfrac{2{\rm Ei}(\veps)}{\log c}\right\rfloor \quad \text{and}\quad	k(\veps)=\left\{\begin{array}{ll}K_{0}&\veps\leq(\sqrt{c}-1)^{-1},\\ c\veps&\veps>(\sqrt{c}-1)^{-1}.\end{array}\right. 
	\end{equation}
	Then, for all $f\in\Hinf(\Sigma_{\psi})$, $\psi\in(0,\frac{\pi}{2})$,
	\begin{equation}\label{eq:MultGeneral}
			\|(fe_{\veps})(A)\| \leq   \left( \pi\: ub_{\left\{\Phi_{n}\right\}_{n=1}^{N_{\veps}}}+m_{\Phi}e^{-k(\veps)}\left(\tfrac{K_{1}}{\log c}+1\right)\right)\|f\|_{\infty,\psi}.
		\end{equation}
		
	\end{enumerate}
	%Here, $m_\phi$, $\kappa_{\Phi}$ and $ub_{\left\{\Phi_{n}\right\}_{n=1}^{N_{\veps}}}$ are defined in (\ref{def:basisconstants}).
	\end{theorem}
	\begin{proof}
	 By \cite[Lem.~9.1.2 and its proof]{haasesectorial}, $A\in\Sect(0)$ with $M(A,\phi)\leq\kappa_{\Phi}M(\psi)$, where $M(\psi)$ only depends on $\psi\in(0,\pi]$. 
	 Clearly, $\sigma(A)\subset[\lambda_{1},\infty)$. This shows \ref{thm:GenMultit1} and \ref{thm:GenMultit2}. 
	 \smallskip
	 
	 To show \ref{thm:GenMultit3}, note that for $N_{\veps}=\lfloor \frac{2{\rm Ei}(\veps)}{\log c}\rfloor$,
	\begin{equation*}
		h(\veps):=c^{N_{\veps}+1}\veps\geq c^{\frac{2{\rm Ei}(\veps)}{\log c}}\veps =e^{2{\rm Ei}(\veps)}\veps\stackrel{(\ref{eq:expintEST})}{\geq} \left(1+\tfrac{1}{\veps}\right)^{e^{-\veps}}\veps>K_{0},
	\end{equation*}
	for some constant $K_{0}\in(0,1)$ and all $\veps>0$.  
	If $N_{\veps}=0$, which means that $2{\rm Ei}(\veps)<\log c$,  then $h(\veps)=c\veps$. 
	Using (\ref{eq:expintEST}), it is easy to see that  $2\:{\rm Ei}\left(\veps_{c}\right)<\log c$ for $\veps_{c}=(\sqrt{c}-1)^{-1}$ and $c\veps_{c}>1>K_{0}$.
	Since $\rm Ei$ is decreasing on $(0,\infty)$, this yields that $h(\veps)\geq k(\veps)$, with $k(\veps)$ defined in (\ref{thm:MultGeneq2}).
	Now, 
	\begin{align*}
		\left\|\:\sum_{n\in\N}f(\lambda_{n})e^{-c^{n}\veps}P_{\left\{n\right\}}\:\right\| \leq{}& \left\|\:\sum_{n=1}^{N_{\veps}}f(c^{n})e^{-c^{n}\veps}P_{\left\{n\right\}}\:\right\| + \left\|\:\sum_{n=N_{\veps}+1}^{\infty}f(c^{n})e^{-c^{n}\veps}P_{\left\{n\right\}}\:\right\|  \\
		\leq{}& \pi\:ub_{\left\{\Phi_{n}\right\}_{n=1}^{N_{\veps}}} \: \left\|fe_{\veps}\right\|_{\infty} + \sum_{\ell=0}^{\infty}\left|f\left(c^{\ell+N_{\veps}+1}\right)e^{-h(\veps)c^{\ell}} \right| \left\|\:P_{\left\{\ell+N_{\veps}+1\right\}}\:\right\|\\
		\leq{}& \pi\: ub_{\left\{\Phi_{n}\right\}_{n=1}^{N_{\veps}}} \: \|f\|_{\infty} +  m_{\Phi}\: \|f\|_{\infty}\: \:\sum_{\ell=0}^{\infty}e^{-k(\veps)c^{\ell}},
	\end{align*}
	where we used the estimate 
	$\left\|\:\sum\nolimits_{n=1}^{N_{\veps}}\lambda_{n}P_{\left\{n\right\}}\:\right\| \leq \pi\:ub_{\left\{\Phi_{n}\right\}_{n=1}^{N_{\veps}}}  \max_{n=1,..,N_{\veps}} |\lambda_{n}|,$ for $\lambda_{n}\in\C$,
	see \cite[Lem.~2.9.1]{Nikolski14}. \newline
	It remains to estimate the sum. By Lemma \ref{le:growthlemma}\:\ref{itGL2},
	\begin{equation*}
		\sum_{\ell=0}^{\infty}e^{-k(\veps)c^{\ell}}\leq e^{-k(\veps)}+\frac{{\rm Ei}(k(\veps))}{\log c}\stackrel{(\ref{eq:expintEST})}{\leq}e^{-k(\veps)}\left(1+\frac{\log\left(1+\tfrac{1}{k(\veps)}\right)}{\log c}\right).
	\end{equation*}
	Since $k(\veps)\geq K_{0}$, we can bound $\log\left(1+\frac{1}{k(\veps)}\right)$ by $K_{1}=\log\left(1+\frac{1}{K_{0}}\right)$.
	\end{proof}
	%\begin{remark}
	%\begin{enumerate}
	%\item
	%We point out that (\ref{eq:MultGeneral}) shows that for $\veps\to\infty$, $\|(fe_{\veps})(A)\|$ goes to $0$ exponentially. 
	%\item Using (\ref{eq:expintEST}), it is easy to show that $\veps_{c}$ can be chosen to be $\frac{1}{\sqrt{c}-1}$ in Theorem \ref{thm:GenMult} \ref{thm:GenMultit3}.
	%\end{enumerate}
	%\end{remark}
	On the the right hand side of (\ref{eq:MultGeneral}), the $\veps$-dependence for small $\veps$ appears only in the term $ub_{\left\{\Phi_{n}\right\}_{n=1}^{N_{\veps}}}$. 
	The following result shows that this indeed exhibits a logarithmic behavior for $\veps\to0$, which confirms the result from Theorem \ref{thm:expstab}, but shows even more, as we will see in Remark \ref{remcor4}.\ref{remcor41}.
	We also show that on Hilbert spaces the behavior is slightly better. 	
	
	\begin{theorem}\label{thm32}
	Let $\left\{\Phi_{n}\right\}_{n\in\N}$, $X$, $c$, $A$ be as in Theorem \ref{thm:GenMult}. Then, the following assertions hold for all $\psi\in(0,\pi)$, $f\in\Hinf(\Sigma_{\psi})$, $\veps>0$. 
	If $X$ is a Banach space, then
		\begin{equation}\label{thm32:eq1}
			\|(fe_{\veps})(A)\| \leq  \left(\frac{K_{2}}{\log c}+1\right) \cdot m_{\Phi}\cdot {\rm Ei}(\veps)\cdot\|f\|_{\infty,\psi}.
		\end{equation}
	If $X$ is a Hilbert space, then 
		\begin{equation}\label{thm32:eq2}
			\|(fe_{\veps})(A)\| \leq  \left(\frac{K_{3}}{\log c}+1\right) \cdot m_{\Phi}\cdot {\rm Ei}(\veps)^{1-\frac{0.32}{\kappa_{\Phi}^{2}}} \cdot \|f\|_{\infty,\psi}.
		\end{equation}
	Here,  $K_{2 }$ and $K_{3}$ are absolute constants.
	\end{theorem}
	\begin{proof}
	By (\ref{eq:MultGeneral}), it remains to estimate $ub_{\left\{\Phi_{n}\right\}_{n=1}^{N_{\veps}}}$. 
	For a basis $\tilde{\Phi}$ of an $N$-dimensional Banach space, it is easy to see that $ub_{\tilde{\Phi}}\leq N m_{\tilde{\Phi}}$. 
	Since $N_{\veps}=\lfloor \frac{2{\rm Ei}(\veps)}{\log c}\rfloor$, and $m_{\left\{\Phi_{n}\right\}_{n=1}^{N_{\veps}}}\leq m_{\Phi}$, this implies (\ref{thm32:eq1}).\newline
	For a basis $\tilde{\Phi}$ of an $N$-dimensional Hilbert space, we have that
	\begin{equation}\label{thm32:proofeq1}
		ub_{\tilde{\Phi}} \leq 2 m_{\tilde{\Phi}} \cdot N^{1-\frac{0.32}{\kappa_{\tilde{\Phi}}^{2}}}.
	\end{equation}
	This is due to a recent result by Nikolski, \cite[Thm.~3.1]{Nikolski14}, which is a slight generalization of a classic theorem by McCarthy-Schwartz, \cite{McCarthySchwartz65}. Hence, because $m_{\left\{\Phi_{n}\right\}_{n=1}^{N_{\veps}}}\leq m_{\Phi}$ and $\kappa_{\left\{\Phi_{n}\right\}_{n=1}^{N_{\veps}}}\leq \kappa_\phi$, 
	\begin{equation*}
	ub_{\left\{\Phi_{n}\right\}_{n=1}^{N_{\veps}}}\leq 2m_{\Phi} N_{\veps}^{1-\frac{0.32}{\kappa_{\Phi}^{2}}}.
	\end{equation*}
	By the definition of $N_{\veps}$,  this yields (\ref{thm32:eq2}). 
	\end{proof}
	\begin{remark}
	The key ingredient of the proof of (\ref{thm32:eq2}) in Theorem \ref{thm32} is the McCarthy-Schwartz-type result, (\ref{thm32:proofeq1}). For general Banach spaces this does not hold. However, there exists a version of McCarthy-Schwartz's result for uniformly convex spaces by Gurarii and Gurarii \cite{Gurarii71}, see also \cite[Thm.~3.6.1 and Cor.~3.6.8]{Nikolski14}. In particular, this enables us to deduce an estimate similar to (\ref{thm32:eq2}) for $L^{p}$-spaces with $p>1$.
	\end{remark}
	\subsection{A particular example}	
	Schauder multipliers of the following type have been used to construct examples in various situations apart from functional calculus, e.g.\ \cite{BenNik99,EisnerZwart06,Haak2012}.%JPP09,ZwartJacobStaffans}.
	\begin{definition}\label{def:L2basis}
	Let $X=L^{2}=L^{2}(-\pi,\pi)$, $\beta\in(\frac{1}{4},\frac{1}{2})$. Define $\left\{\Phi_{n}\right\}_{n\in\N}$ by
	\begin{equation*}
	\Phi_{2k}(t)=w_{\beta}(t)e^{ikt}, \quad \Phi_{2k+1}(t)=w_{\beta}(t)e^{-ikt},
	\end{equation*}
	where $k\in\N\cup\left\{0\right\}$, $t\in(-\pi,\pi)$ and 
	\begin{equation*}
	w_{\beta}(t)=\left\{\begin{array}{ll}|t|^{\beta},& |t|\in(0,\frac{\pi}{2}),\\ (\pi-|t|)^{-\beta}, & |t|\in[\frac{\pi}{2},\pi).\end{array}\right.
	\end{equation*}
	\end{definition}
	$\left\{\Phi_{n}\right\}_{n\in\N}$ forms a Schauder basis of $L^{2}$, see Lemma \ref{le:settingexample}. 	 
		\begin{theorem}\label{sharpnesslog}
		There exist  $g\in\Hinf(\Cp)$ and constants $K_{i}>0$, $i\in\{0,1,..,4\}$ such that the following holds. For every $\delta\in\left(0,\tfrac{1}{2}\right)$ there exists $A\in \Sect(0)$ on $X=L^{2}(-\pi,\pi)$ with
		\begin{enumerate}[label={(\roman*)}]
		\item\label{thm:sharpit1} $0\in\rho(A)$ and ${\rm dist}(\sigma(A),0)=2$,
		\item\label{thm:sharpit2} $\frac{K_{3}}{\delta}\leq M(A,\psi)\leq \frac{K_{4}}{\delta}M(\psi)$ for all $\psi\in(0,\pi]$, where $M(\psi)$ only depends on $\psi$.
				\item\label{thm:sharpit4}
		For all $\veps>0$, $f\in\Hinf(\Cp)$, and some absolute constant $K_{0}$, 
		\begin{equation}\label{2eq:Thm2}
			\|(fe_{\veps})(A)\| \leq K_{1} \cdot \tfrac{1}{\delta} \cdot {\rm Ei}(\veps )^{1-K_{0}\delta^{2}}\cdot \|f\|_{\infty}.
		\end{equation}
		 \item\label{thm:sharpit3} For $\veps<\tfrac{1}{4}$,
		\begin{equation}\label{1eq:Thm2}
		\|(ge_{\veps})(A)\| \geq K_{2}\cdot \tfrac{1}{\delta} \cdot |\log(\veps )|^{1-\delta}.
		\end{equation}

		\end{enumerate}
		\end{theorem}
		\begin{proof}
		Let $\beta=\frac{1}{2}-\frac{\delta}{4}\in(\frac{3}{8},\frac{1}{2})$ and let $\left\{\Phi_{n}\right\}_{n\in\N}$ denote the basis from Definition \ref{def:L2basis}. %and $\left\{\Phi_{n}^{*}\right\}_{n\in\N}$ its dual basis, see Lemma \ref{le:settingexample}. 
		By Lemma \ref{le:settingexampleit1}, $\kappa_{\Phi}\sim\frac{1}{1-2\beta}=\frac{2}{\delta}$. With respect to $\left\{\Phi_{n}\right\}_{n\in\N}$, we consider the multiplication operator $A=\mathcal{M}_{\lambda}$ on $L^{2}(-\pi,\pi)$, where $\lambda_{n}=2^{n}$. 
		By Theorem \ref{thm:GenMult}, \ref{thm:sharpit1} and  the inequality $M(A,\psi)\leq \frac{K_{4}}{\delta}M(\psi)$ in \ref{thm:sharpit2} follow. The other inequality in \ref{thm:sharpit2} will be discussed at the end of the proof.
		
		\ref{thm:sharpit4} follows by (\ref{thm32:eq2}) from Theorem \ref{thm32}.
			
			To show \ref{thm:sharpit3}, we choose $x(t)= |t|^{-\beta}\mathbb{1}_{(0,\frac{\pi}{2})}(|t|)$ and $y(t)= (\pi-|t|)^{-\beta}\mathbb{1}_{(\frac{\pi}{2},\pi)}(|t|)$. By Lemma \ref{le:settingexampleit4}, we have that for $x=\sum_{n} x_{n}\Phi_{n}$ and $y=\sum_{n} y_{n}\Phi_{n}^{*}$, the coefficients $x_{n}$ and $y_{n}$ are real and
			\begin{equation}\label{eq:coeffbeh99}
			c_{3}\tfrac{k^{-1+2\beta}}{1-2\beta}\leq x_{2k}= x_{2k+1}\leq C_{3} \tfrac{k^{-1+2\beta}}{1-2\beta}\quad  \text{ and}\quad y_{2k}=y_{2k+1}=(-1)^{k}2\pi\cdot x_{2k}.
			\end{equation}
			By setting $\mu_{2n}=\mu_{2n+1}=(-1)^{n}$ for all $n\in\N$ and using that $\langle \Phi_{n},\Phi_{m}^{*}\rangle=\delta_{nm}$, we conclude that
		\begin{align}
		|\langle \mathcal{M}_{\mu}\mathcal{M}_{e^{-\lambda_{n}\veps}}x,y\rangle| = {}& 2\pi\sum_{n\in\N} e^{-\lambda_{n}\veps} |x_{n}|^{2}\notag\\ 
		\stackrel{(\ref{eq:coeffbeh99}) }{\geq}  {}& \tfrac{c_{3}^{2}}{(1-2\beta)^{2}}\sum_{k\in\N} (e^{\lambda_{2k}\veps}+e^{\lambda_{2k+1}\veps}) k^{-2+4\beta}
		\stackrel{\eqref{eq:growthlemma2}}{\geq} {}\tfrac{c_{3}^{2}c_{-\frac{1}{2},4}}{(1-2\beta)^{2}} |\log(\veps)| ^{-1+4\beta}, \label{eq2:Thm2}
		\end{align}
		for $\veps<\frac{1}{4}$, where we used Lemma \ref{le:growthlemma}~\ref{itGL1}. Since $\|x\|\cdot\|y\|\sim\frac{1}{1-2\beta}$, and $2-4\beta=\delta$,
		\begin{equation*}
		\|\mathcal{M}_{\mu}\mathcal{M}_{e^{-\lambda_{n}\veps}}\|\geq K_{2}\: \frac{1}{\delta}\: |\log(\veps)|^{1-\delta},\quad \veps<\tfrac{1}{4}.
		\end{equation*}
		Since $(\lambda_{n})$ is an interpolating sequence, we can find $g\in\Hinf(\Cm)$ such that $g(\lambda_{n})=\mu_{n}$ for all $n\in\N$. Thus, $g(A)=\mathcal{M}_{\mu}$ and (\ref{1eq:Thm2}) follows.
		
		 To see that $\frac{K_{3}}{\delta}\leq M(A,\phi)$, one can show that $|\langle R(-1,A)x,y\rangle|\geq \frac{K_{3}}{\delta}\|x\|\:\|y\|$ for $x$, $y$ from above with a similar proof as for \ref{thm:sharpit3}.
				\end{proof}
				Theorem \ref{sharpnesslog} implies that estimate (\ref{eq11:expstab}) in Theorem \ref{thm:expstab} is sharp in $M(A,\phi)$ and $\veps$ as $\delta\to0^{+}$.
				\begin{corollary}\label{cor4}
				Let $X$ be an infinite-dimensional Hilbert space and  $\phi<\frac{\pi}{2}$. Then, there exists  $K_{1},K_{2}>0$ and a sequence $(A_{n})_{n\in\N}\subset \Sect(0)$ with ${\rm dist}(\sigma(A_{n}),0)=2$ such that for all $\veps <\frac{1}{4}$,
				\begin{equation}\label{eq:cor4}
				K_{2}\:|\log\veps\,|<\sup\left\{\frac{\|(fe_{\veps})(A_{n})\|}{M(A_{n},\phi)\:\|f\|_{\infty}}: f\in\Hinf(\Cp)\setminus\{0\}, n\in\N\right\}<K_{1}\:|\log\veps\,|.
				\end{equation}
				%	\sup\left\{\frac{\|(fe_{\veps})(A)\|}{M(A,\phi)\:\|f\|_{\infty}}: {\rm dist}(\sigma(A),0)\geq1,f\in\Hinf(\Cp)\right\}>K\:|\log\veps\,|.
				\end{corollary}
				\begin{remark}\label{remcor4}
				\begin{enumerate}
				\item\label{remcor41} In Corollary \ref{cor4}, the distance ${\rm dist}(\sigma(A_{n}),0)$ is fixed. Such a condition is not surprising in the view of Theorem \ref{thm:expstab} because of the following reasoning. If ${\rm dist}(\sigma(A_{n}),0)$ tends to $0$ as $n\to\infty$, then, by $\|A_{n}^{-1}\|\geq \frac{1}{{\rm dist}(A_{n},0)}$, it follows that $\frac{1}{\|A_{n}^{-1}\|}\to 0$ and thus,  the right-hand side of \eqref{eq11:expstab} would tend to $\infty$ for fixed $\veps$. \newline 
	However, it is interesting that for $A_{n}$ as chosen in the proof of Theorem \ref{sharpnesslog} (more precisely, $A_{n}=A$ and $\delta \sim n^{-1}$) one can show that $\|A_{n}^{-1}\|\sim n$, thus, $\frac{1}{\|A_{n}^{-1}\|}\to 0$. This shows that the upper estimate in \eqref{eq:cor4} does not follow from Theorem \ref{thm:expstab}, but instead is due to Theorem \ref{thm:GenMult}. Hence, the latter result can be seen as an improvement in the situation of Schauder multipliers for which the dependence on $\|A_{n}^{-1}\|^{-1}$ can be replaced by ${\rm dist}(\sigma(A_{n}),0)$.
				\item
				Corollary \ref{cor4} shows that the logarithmic behavior in $\veps$ is essentially optimal.
				However, we point out that in Theorem \ref{sharpnesslog}, $M(A,\phi)\to\infty$ as $\delta\to0^{+}$. 
				Therefore, for fixed $M(A,\phi)$, the behavior in $\varepsilon\to0^{+}$ could be better than $|\log\veps|$.
				For a similar effect we refer to the question of  sharpness of Spijker's result on the \textit{Kreiss-Matrix-Theorem}, see \cite{Spijker91,SpijkerTracognaWelfert03} and the recent contribution by Nikolski \cite{Nikolski14}.
				\item
				In \cite[Thm.~2.1, Rem.~2.2]{Vitse05}, it is shown that estimate (\ref{cor1:eq}) is indeed sharp in $\veps$ and $\sigma$ on general Banach spaces. 
				Furthermore, Vitse \cite[Thm.~2.3 and Rem.~2.4]{Vitse05} states that for any Hilbert space and any $\delta\in(0,1)$, one can find a sectorial operator $A$ with $\omega_{A}<\frac{\pi}{2}$ such that 	
					\begin{equation}
					\sup\left\{\|g(A)\|:g\in H^{\infty}[\veps,\sigma],\|g\|_{\infty,\Cp}\leq 1\right\} \geq a \log \left(\frac{e\sigma}{\veps}\right)^{\delta},
					\end{equation}
					where $a$ depends only on $M(A,\frac{\pi}{2})$.
					Therefore, item \ref{thm:sharpit4} of Theorem \ref{sharpnesslog} and Corollary \ref{cor4} can be seen as a version for $0\in\rho(A)$ and $\sigma=\infty$.
					  Theorem \ref{sharpnesslog}~\ref{thm:sharpit3} shows that the behavior of $\|(fe_{\veps})(A)\|$ is indeed better than $|\log(\veps)|$. 
					We remark that Vitse's result, \cite[Thm.~2.3]{Vitse05} is stated for Banach spaces which \textit{uniformly contain uniformly complemented copies of $\ell^{2}$}, which is more general than for Hilbert spaces. It is not hard to see that Corollary \ref{cor4} generalizes to this more general setting.
			\end{enumerate}
				\end{remark}
			\section{Square function estimates improve the situation}\label{sec:Squarefunct}
			The following notion characterizes bounded $\Hinf$-calculus on Hilbert spaces. It was already used in the early work of McIntosh, \cite{mcintoshHinf} and has been investigated intensively since then.
				\begin{definition}\label{def:sqfctest}
				Let $A\in\Sect(\omega)$ on the Banach space $X$. We say that $A$ satisfies \textit{square function estimates} if there exists $\zeta\in\HinfO(\Sigma_{\phi})\setminus\left\{0\right\}$, $\phi>\omega$ and $K_{\zeta}>0$ such that 
				\begin{equation}\label{sqfctest}
				\int_{0}^{\infty} \|\zeta(tA)x\|^{2}\ \frac{dt}{t}\leq K_{\zeta}^{2} \|x\|^{2},\quad \forall x\in X.
				\end{equation}
				\end{definition}
				The property of satisfying square functions estimates does not rely on the particular 
				function $\zeta$. In fact, for $\zeta, \eta\in\HinfO(\Sigma_{\phi})\setminus\left\{0\right\}$
				\begin{equation}\label{eq1:sqfctest}
				\exists K>0 \ \forall h\in\Hinf(\Sigma_{\phi}):\quad \int_{0}^{\infty} \|(\zeta_{t}h)(A)x\|^{2}\ \frac{dt}{t} \leq K^{2} \|h\|_{\infty,\phi}^{2} \int_{0}^{\infty} \|\eta_{t}(A)x\|^{2}\ \frac{dt}{t}, 
				\end{equation}
				where $\zeta_{t}(z)=\zeta(tz)$ and $\eta_{t}(z)=\eta(tz)$. 
				We remark that $K$ can be chosen only depending on $\zeta,\eta$ and $M(A,\phi)$.
				The result can be found in \cite[Prop.~E]{McIntoshOpHam} for Hilbert spaces, 
				but also holds for general Banach spaces as pointed out in \cite[Satz 2.1]{haakthesis}.	
				The following result goes back to McIntosh in his early work on $\Hinf$-calculus, \cite{mcintoshHinf} and can also be found in \cite[Thm.~7.3.1]{haasesectorial}.
				\begin{theorem}[McIntosh '86]\label{thm:McIntosh}
				Let $X$ be a Hilbert space, $A\in\Sect(\omega)$, densely defined and with dense range. Then, the following assertions are equivalent.
				\begin{enumerate}
				\item The $\Hinf(\Sigma_{\mu})$-calculus for $A$ is bounded for some (all) $\mu\in(\omega,\pi)$.
				\item $A$ and $A^{*}$ satisfy square function estimates.
				\end{enumerate}
				\end{theorem}
				Note that on a Hilbert space, $\overline{D(A)}=X$ follows from sectorality, see \cite[Prop.~2.1.1]{haasesectorial}.\newline
				 Le Merdy showed in \cite[Thm.~5.2]{LeMerdy2003} that having square function estimates for only $A$ or $A^{*}$ is not sufficient to get a bounded calculus.
				However, we will show that the validity of single square function estimates always yields an improved growth of $\|(fe_{\veps})(A)\|$ near zero. Roughly speaking, having `half of the assumptions' in McIntosh's result indeed interpolates the general logarithmic behavior of $\|(fe_{\veps})(A)\|$.
				\begin{theorem}\label{thm:sqfctestlog}
				Let $\omega<\phi<\frac{\pi}{2}$ and $A\in\Sect(\omega)$ be densely defined on the Banach space $X$. Assume that
				\begin{itemize}
				\item $0\in\rho(A)$ and that
				\item $A$ satisfies square function estimates.
				\end{itemize}
				Then for every $\kappa\in(0,1)$ there exists $C=C(\kappa,M(A,\phi),\cos(\phi))>0$ such that for all $\veps>0$ and  for $f\in\Hinf(\Sigma_{\phi})$,
				\begin{equation}\label{sqrtlogeq1}
				\|(fe_{\veps})(A)\| \leq CK_{\eta} \cdot \left[{\rm Ei}\left(\frac{\kappa\veps\cos\phi}{\|A^{-1}\|}\right)\right]^{\frac{1}{2}}\cdot \|f\|_{\infty,\phi},
				\end{equation}
			  where $K_\eta$ denotes the constant in (\ref{def:sqfctest}) for $\eta(z)=z^{\frac{1}{2}}e^{-z}$.
				
				\end{theorem}
				
				\begin{proof} 
				Let $\zeta(z)=ze^{-z}$. Since $\sqrt{\zeta}\in\Hinf_{0}(\Sigma_{\phi})$, we have by (\ref{eq1:sqfctest})  that			
				\begin{equation}\label{sqrtlogeq2}
				\int_{0}^{\infty}\|(fe_{\veps}\sqrt{\zeta_{t}})(A)x\|^{2} \ \frac{dt}{t}\leq K^{2}\  \|fe_{\veps}\|_{\infty,\phi}^{2}\cdot \int_{0}^{\infty}\|(\sqrt{\zeta_{t}})(A)x\|^{2}\ \frac{dt}{t},
				\end{equation}
				where $K>0$ only depends on $M(A,\phi)$.
				The integral on the right-hand side is finite because $A$ satisfies square function estimates (for $\eta=\sqrt{\zeta}$).
				 It is easy to see that 
				$\int_{0}^{\infty}\zeta_{t}(z)\ \frac{dt}{t}=1$ for $z\in\Sigma_{\phi}$,  and applying the Convergence Lemma (to Riemann sums), \cite[Prop.~5.1.4]{haasesectorial}, yields $y=\int_{0}^{\infty}\zeta_{t}(A)y\frac{dt}{t}$ for $y\in X$.
				Thus, 
				\begin{align}\notag
				\|(fe_{\veps})(A)x\| = {}&\left\| \int_{0}^{\infty} (fe_{\veps}\zeta_{t})(A)x\ \frac{dt}{t}\right\|\\
				\leq{}&\int_{0}^{\infty} \left\| \left(e_{\frac{\veps}{2}}\sqrt{\zeta_{t}}\right)(A) \left(fe_{\frac{\veps}{2}}\sqrt{\zeta_{t}}\right)(A)x\right\|\frac{dt}{t}\notag\\
				\leq {}& \left(\int_{0}^{\infty}\|\left(e_{\frac{\veps}{2}}\sqrt{\zeta_{t}}\right)(A)\|^{2}\frac{dt}{t}\right)^{\frac{1}{2}}\left(\int_{0}^{\infty}\|\left(fe_{\frac{\veps}{2}}\sqrt{\zeta_{t}}\right)(A)x\|^{2}\frac{dt}{t}\right)^{\frac{1}{2}}.\label{sqrtlogeq3}
				\end{align}
				In the last step we used that $t\mapsto(e_{\frac{\veps}{2}}\sqrt{\zeta_{t}})(A)$ is continuous in the operator norm which makes the first integral exist. In fact, $e^{-\frac{\veps z}{2}}\sqrt{\zeta_{t}(z)}=(zt)^{\frac{1}{2}}e^{-z\frac{t+\veps}{2}}\in\HinfO(\Sigma_{\phi})$, and hence by the functional calculus for sectorial operators,
				\begin{equation}\label{sqrtlogeq4}
				\left[e^{-\frac{\veps z}{2}}\sqrt{\zeta_{t}(z)}\right](A)=t^{\frac{1}{2}}A^{\frac{1}{2}}T\left(\frac{t+\veps}{2}\right).
				\end{equation}
				For $s>0$ we have that $A^{\frac{1}{2}}T(s)=A^{-\frac{1}{2}}AT(s)=A^{-\frac{1}{2}}\frac{\partial}{\partial s}T(s)$. 
				Since  $s\mapsto T(s)$ is $C^{\infty}(\Rp,\Bo(X))$ for analytic semigroups 
				and $A^{-\frac{1}{2}}\in\Bo(X)$ as $0\in\rho(A)$, 
				we get indeed that $t\mapsto(e_{\frac{\veps}{2}}\sqrt{\zeta_{t}})(A)$
				 is continuous in the operator norm.\newline
				 By (\ref{sqrtlogeq2}) we can estimate the second integral in (\ref{sqrtlogeq3}) and find
					\begin{equation}\label{sqrtlogeq45}
				\|(fe_{\veps})(A)x\| \leq \left(\int_{0}^{\infty}\|\left(e_{\frac{\veps}{2}}\sqrt{\zeta_{t}}\right)(A)\|^{2}\frac{dt}{t}\right)^{\frac{1}{2}} \cdot K\cdot  \|f\|_{\infty,\phi} \cdot K_{\sqrt{\zeta}} \|x\|.
					\end{equation}
				Hence, it remains to study the first term in (\ref{sqrtlogeq45}). 
				By (\ref{sqrtlogeq4}) and Lemma \ref{le:fracT(t)growth},
				\begin{align}
				\int_{0}^{\infty}\|\left(e^{-\frac{\veps}{2}\cdot}\sqrt{\zeta_{t}}\right)(A)\|^{2}\frac{dt}{t}={}& \int_{\frac{\veps}{2}}^{\infty} \|A^{\frac{1}{2}}T(t)\|^{2}dt
				\leq{}\tilde{C}^{2} \int_{\frac{\veps}{2}}^{\infty} \frac{e^{-2tR\kappa\cos\omega}}{t}dt
				={}\tilde{C}^{2}\:{\rm Ei}\left(\kappa\veps R\cos\phi\right)\label{sqrtlogeq5},
				\end{align}
				for $\kappa\in(0,1)$, $R=\frac{1}{\|A^{-1}\|}$ and $\tilde{C}=C_{\frac{1}{2},\kappa} M(A,\phi)(\cos\phi)^{-\frac{1}{2}}>0$,  see Lemma \ref{le:fracT(t)growth}.
				\end{proof}
				\begin{remark}\label{rem:Yakubovich}
				In \cite{Yakubovich2011}, Gal\'e, Miana and Yakubovich draw a connection  between the $H^{\infty}$-calculus for sectorial operators and the theory of functional models for Hilbert space operators.
				In addition, they prove a \textit{logarithmic gap} (as they call it) between the Hilbert space $X$ and $X_{A}$. $X_{A}$ is the space of elements of $X$ such that
				\begin{equation*}
					\|x\|_{A}^{2}=\int_{0}^{\infty}\|\zeta(tA)x\|^{2}\frac{dt}{t}<\infty,
				\end{equation*}
				for some $\zeta\in\HinfO(\Sigma_{\phi})\setminus \left\{0\right\}$.
				One can show that the $H^{\infty}(\Sigma_{\phi})$-calculus is bounded if and only if the norm $\|\cdot\|_{A}$ is equivalent to the norm of the space $X$, see \cite{Yakubovich2011} and the references therein.
				%Loosely speaking $H_{A}$, is a space on which one has a bounded calculus. 
				 In the view of \cite[Sec.~6.4]{haasesectorial}, $X_{A}$ is the \textit{intermediate space} $X_{0,\zeta,2}$. This space, in turn, can be shown to be equal to the real interpolation space $\left(X^{(1)}, X^{(-1)}\right)_{\frac{1}{2},2}$\:, see \cite[Thm.~6.4.5]{haasesectorial}, where $X^{(1)}$ and $X^{(-1)}$ are the \textit{homogeneous spaces} for $A$.			 
				In \cite{Yakubovich2011}, the \textit{logarithmic gap} refers to the result that 
				\begin{equation}\label{eq:loggap}
					\forall r>\tfrac{1}{2}\ \exists\ c_{r}>0:\quad c_{r}^{-1}\|\Lambda_{1}(A)^{-r}x\| \leq \|x\|_{A}\leq c_{r}\|\Lambda_{1}(A)^{r}x\|,
				\end{equation}
				for all $x\in\Lambda_{1}(A)^{-r}X$, where $\Lambda_{1}(z)=Log(z)+2\pi i$ (here, $Log$ denotes the principal branch of the logarithm) and where $\Lambda_{1}^{-r}(A)X$ is interpreted as a (dense) subspace of $X$, see Theorem 2.1 in \cite{Yakubovich2011}. We learned from D.~Yakubovich that this result can be used to derive estimates of $\|(fe_{\veps})(A)\|$ of the form in (\ref{eq:int1}), which are slightly weaker than our results presented here. \newline
				However, as $X_{A}$ is an interpolation space, \eqref{eq:loggap} should be rather seen as the consequence of the `idea' that functional calculus properties for $A$ improve in the corresponding interpolation spaces, see \cite[Sec.~6.5]{haasesectorial}.
				More generally, this motivates the study of the relation between our results and interpolation spaces. This is subject to future research.
				\end{remark}

				The following theorem proves that the result in Theorem \ref{thm:sqfctestlog} is essentially sharp.
				\begin{theorem}
				There exist a Hilbert space $X$, $g\in\Hinf(\Cp)$ and $K_{0}>0$ such that for any $\delta\in(0,\frac{1}{4})$ there exists $A\in\Sect(0)$ on $X$ with the following properties.
				\begin{enumerate}[label={(\roman*)}]
				\item\label{thm:4.5i}  $0\in\rho(A)$ and ${\rm dist}(0,\sigma(A))=2$.
				\item $A^{*}$ satisfies square function estimates.
				\item \label{thm:4.5iii} For $\veps\in(0,\tfrac{1}{4})$,
				\begin{equation}
				\|(ge_{\veps})(A)\| \geq K_{0}\cdot |\log(\veps)|^{\frac{1}{2}-\delta} .
				\end{equation}
					\item  \label{thm:4.5iv}
		There exists $c_{\delta}$ such that for all $\veps>0$ and $f\in\Hinf(\Cp)$,
		\begin{equation}\label{2eq:Thmsqft}
			\|(fe_{\veps})(A)\| \leq c_{\delta}\cdot {\rm Ei}(\veps)^{\frac{1}{2}-\frac{\delta}{6}}\cdot \|f\|_{\infty},
		\end{equation}
				\end{enumerate}
				\end{theorem}
				\begin{proof}
				%The example is a multiplication operator with respect to to a Schauder basis.
				Let us consider $X=L^{2}(-\pi,\pi)$, $\beta\in(\frac{5}{12},\frac{1}{2})$ and the basis $\{\Psi_{n}\}_{n\in\N}$ from Lemma \ref{le:settingexampleit6}. Let $A$ be $\mathcal{M}_{\left\{2^{n}\right\}}$ with respect to $\{\Psi_{n}\}_{n\in\N}$. 
				By Theorem \ref{thm:GenMult}, $A\in\Sect(0)$ and \ref{thm:4.5i} holds.\smallskip
				
				It is not hard to see that $A^{*}$ equals the multiplication operator $\mathcal{M}_{\left\{2^{n}\right\}}$ with respect to the basis $\{\Psi_{n}^{*}\}_{n\in\N}$ defined in Lemma \ref{le:settingexampleit6}~\ref{le:A4ii}.  It is well-known that if the basis is Besselian, then $\mathcal{M}_{\left\{2^{n}\right\}}$ with respect to this basis satisfies square function estimates, see e.g.\ \cite[Proof of Thm.~5.2]{LeMerdy2003}. By the Lemma, $\{\Psi_{n}^{*}\}_{n\in\N}$ is Besselian, hence, $A^{*}$ satisfies square function estimates.
				\smallskip 

	Let $x=\sum_{n}x_{n}\Psi_{n}$ and $y=\sum_{n}y_{n}\Psi_{n}^*$ be as in Lemma \ref{le:settingexampleit6}~\ref{le:A4iv}. By the form of $\{y_{n}\}$ and since $(2^{n})_{n\in\N}$ is interpolating, we find $g\in\Hinf(\Cp)$ (independent of $\beta$) such that $g(2^{n})={\rm sgn}(y_{n})$ for all $n\in\N$. Hence, since $\langle \Psi_{n},\Psi_{m}^{*}\rangle=\delta_{mn}$,
	\begin{align}\notag
	\langle (g e_{\veps})(A)x,y\rangle = {}&\sum_{n\in\N} g(2^{n})e^{-2^{n}\veps} x_{n}y_{n}=\sum_{n\in\N} e^{-2^{n}\veps}|x_{n}y_{n}|\\
	\geq{}&\tfrac{c_{3}c_{4}}{1-2\beta}\sum_{k\in\N}(e^{-2^{2k}\veps}+e^{-2^{2k+1}\veps})k^{-2+3\beta}\stackrel{\eqref{eq:growthlemma2}}{\geq} \tfrac{c_{3}c_{4}c_{-3/4,4}}{1-2\beta}|\log(\veps)|^{-1+3\beta},
	\end{align}
	where we used Lemma \ref{le:settingexampleit6}~\ref{le:A4iv}, and Lemma \ref{le:growthlemma}~\ref{itGL1} noting that $-2+3\beta\in(-\frac{3}{4},0)$.  Since $\|x\|_{L^{2}}\:\|y\|_{L^{2}}\sim\frac{1}{1-2\beta}$ and
	by defining $\beta=\frac{1}{2}-\frac{\delta}{3}$,  assertion \ref{thm:4.5iii} follows.\smallskip
	
		To show \ref{thm:4.5iv}, let  $x=\sum x_{n}\Psi_{n}$, $y=\sum y_{n}\Psi_{n}^{*}$ be general elements of $X$. For $f\in\Hinf(\Cp)$, 
		\begin{align}
		\langle (fe_{\veps})(A)x,y\rangle 
		%{}& \langle \sum_{n\in\N} f(2^{n})e^{-2^{n}\veps}x_{n}\Psi_{n},\sum_{n\in\N}y_{n}\Psi_{n}^{*} \rangle\notag
				=\sum \nolimits_{n\in\N} f(2^{n})e^{-2^{n}\veps} x_{n}y_{n},\label{eq1:Thm2}
		\end{align}
		where we used that $\langle \Psi_{n},\Psi_{m}^{*}\rangle=\delta_{nm}$. By the Cauchy-Schwarz inequality,
		\begin{equation*}
		|\langle (fe_{\veps})(A)x,y\rangle | \leq \|f\|_{\infty}\cdot \|(e^{-2^{n}\veps/2}x_{n})\|_{2}\cdot \|(e^{-2^{n}\veps/2}y_{n})\|_{2}.
		\end{equation*}
		Since  $\left\{\Psi_{n}^{*}\right\}_{n\in\N}$ is Besselian, the uniform boundedness principle implies that  there exists a constant $C_{\beta}>0$ such that $\|(y_{n})\|_{2}\leq C_{\beta} \|y\|$ for all $y\in X$. Therefore,
		\begin{equation}
		\label{eq12:Thm42}
		|\langle (fe_{\veps})(A)x,y\rangle | \leq C_{\beta}\:\|f\|_{\infty}\: \|(e^{-2^{n-1}\veps}x_{n})\|_{2}\: \|y\|.
		\end{equation}
By (\ref{eq23:Thm42}) in Lemma \ref{le:settingexampleit6}~\ref{le:A4v}, 
		\begin{align*}
		|\langle (fe_{\veps})(A)x,y\rangle | \leq {}& C_{\beta}\: C_{6}\:\|(n^{\beta-1})\|_{\frac{3-2\beta}{4}}\:\|f\|_{\infty}\: {\rm Ei}(\veps)^{\frac{1+2\beta}{4}} \: \|x\|\: \|y\|.
		\end{align*}
		Substituting $\beta=\frac{1}{2}-\frac{\delta}{3}$ and $c_{\delta}:=C_{6}\:C_{\beta}\:\|(n^{\beta-1})\|_{\frac{3-2\beta}{4}}$ yields (\ref{2eq:Thmsqft}).
		\end{proof}

 \section{Discussion and Outlook}
\label{discussion}
	\subsection{Comparison with a result of Haase $\&$ Rozendaal}
		In \cite{HaaseRozendaal13} Haase and Rozendaal derived a result of the type of Theorem \ref{thm1} for Hilbert spaces, but for general bounded, not necessarily analytic, $C_0$-semigroups. 
		We devote this subsection to compare the results, in particular the dependence on the semigroup bound and the sectorality constant, respectively.
		We define the right half-plane ${\rm R}_{\delta}=\left\{z\in\C:\Re z>\delta\right\}$. 
		Using transference principles developed by Haase in \cite{haasetransference11}, the following result was proved in \cite{HaaseRozendaal13}.
		\begin{theorem}[Haase, Rozendaal, Corollary 3.10 in \cite{HaaseRozendaal13}]\label{thm:haaserozendaal}
			Let $H$ be a Hilbert space and $-A$ generate a bounded semigroup $T$ on $H$ and define $B=\sup_{t>0} \| T(t)\|$.
			Then, there exists an absolute constant $c>0$ such that for all $\veps,\delta>0$ the following holds. \newline
			For $f\in \Hinf({\rm R}_{\delta})$, the operator $(fe_{\veps})(A)=f(A)T(\veps)$ is bounded and
				\begin{equation}\label{eq:HaaseRozendaal}
					\|(fe_{\veps})(A)\| \leq B^{2} \cdot\eta(\delta,\veps)\cdot \|f\|_{\infty,{\rm R}_{\delta}},
				\end{equation}
				where 	
				\begin{equation*}
					\eta(\delta,\veps)=\left\{\begin{array}{ll}c |\log(\veps\delta)|, &\delta\veps\leq\tfrac{1}{2},\\
					2c,&\delta\veps>\tfrac{1}{2}.\end{array}\right.
				\end{equation*}
		\end{theorem}	
		We can now compare Theorems \ref{thm1} and \ref{thm:haaserozendaal} by setting $r_{0}=\delta$. 
		Then $\Omega_{\phi,\delta}\subset {\rm R}_{\delta}$  for all $\phi\in(0,\tfrac{\pi}{2}]$ and  thus, for functions $f\in\Hinf({\rm R}_{\delta})$, we have $\|f\|_{\infty,\Omega_{\phi,\delta}} \leq  \|f\|_{\infty,{\rm R}_{\delta}}$.
		 Hence, Theorem \ref{thm1} yields
		\begin{equation}\label{eq:ourconstant}
			\|(fe_{\veps})(A)\| \leq M(A,\phi) \cdot b(\veps,\delta,\phi)\cdot \|f\|_{\infty,{\rm R}_{\delta}},
		\end{equation}
		for all $\phi\in(\omega_{A},\tfrac{\pi}{2})$ and $f\in\Hinf({\rm R}_{\delta})$, where 
		\begin{equation*}
				b(\veps,\delta,\phi) \sim \left\{ \begin{array}{ll}|\log(\veps\delta \cos\phi)|,&  \veps \delta<\tfrac{1}{2},\\  |\log\tfrac{\cos\phi}{2}|,&\veps\delta\geq\tfrac{1}{2}.\end{array}\right.
			\end{equation*}
		Let us collect the key observations when comparing  (\ref{eq:HaaseRozendaal}) and (\ref{eq:ourconstant}).
		\begin{enumerate}
			\item The square of the semigroup bound $B$  gets replaced by the $M(A,\phi)$ in our result.
			\item Our estimate depends on another parameter $\phi$ that accounts for the fact that the spectrum is truly lying in a sector rather than the half-plane. Taking the infimum over all $\phi\in(\omega_{A},\tfrac{\pi}{2})$ in (\ref{eq:ourconstant}) yields an optimized estimate. However, then the constant dependence on $M(A,\phi)$  becomes unclear.
			See also Theorem \ref{thm2}.
			\item The dependence on $\phi$ also explains how the estimate explodes when considering operators $A$ with sectorality angle $\omega_{A}$ tending to $\tfrac{\pi}{2}$.  However, one can cover this behavior in terms of the constant $M=M(A,\tfrac{\pi}{2})$: Taking $\phi=\arccos\tfrac{1}{2M}$, we get by Lemma \ref{lemVitse} that $M(A,\phi)\leq 2 M$ and thus (\ref{eq:ourconstant}) becomes
		\begin{equation}\label{eq:ourconstant2}
			\|(fe_{\veps})(A)\| \leq M \cdot b(\veps,\delta,\arccos\tfrac{1}{2M})\cdot \|f\|_{\infty,{\rm R}_{\delta}}.
		\end{equation}
		Therefore, we get an $M$-dependence of the form $\mathcal{O}(M(\log(M)+1))$. 
		
             \item By Theorem \ref{cor1}, the semigroup bound of $e_{t}(A)$ is also of order  $\mathcal{O}(M(\log(M)+1))$. Whether $B\sim M (\log(M)+1))$ in general is still an open problem, see also \cite[Rem.~1.3]{Vitse05}. However, it is easy to see that, in general, $M(A,\pi)\leq B$. Therefore, for an absolute constant $K>0$,
			\begin{equation}
				M(A,\pi) \leq B \leq K\:M(\log(M)+1).
			\end{equation}
		\end{enumerate}
	\subsection{Besov calculus}
	We briefly introduce the following homogenous Besov space and refer to \cite[Section 1.7]{Vitse05} and the references therein for details, see also \cite{haasetransference11}. The notation follows \cite{Vitse05}.
	The space $B_{\infty,1}^{0}$ can be defined as the space of holomorphic functions $f$ on $\Cp$ such that
		\begin{equation*}
			\|f\|_{B}:= \|f\|_{\infty} + \int_{0}^{\infty} \|f'(t+i\cdot)\|_{\infty} dt<\infty.
		\end{equation*}
		Clearly, $B_{\infty,1}^{0}$, equipped with the above norm, is continuously embedded in $\Hinf(\Cp)$. 
		Moreover, $\cup _{0<\veps<\sigma}H^{\infty}[\veps,\sigma]$, see Section \ref{sec:Vitse}, lies dense in $B_{\infty,1}^{0}$ and the following norm is equivalent to $\|\cdot\|_{B}$, see \cite[Thm.~A.1]{Vitse05},
			\begin{equation*}
				\|f\|_{\ast B}=|f(\infty)|+\sum_{k\in\mathbb{Z}}\|f\ast\hat{h_{k}}\|_{\infty},
			\end{equation*}
			where $h_{k}$ is the continuous, triangular-shaped function that is linear on the intervals $[2^{k-1},2^{k}]$ and $[2^{k},2^{k+1}]$, vanishes outside $[2^{k-1},2^{k+1}]$, and such that $h_{k}(2^{k})=1$.
			Thus, $\left\{h_{k}\right\}_{k\in\N}$ is a partition of unity with $\sum_{k\in\mathbb{Z}}h_{k}\equiv1$ locally finite on $(0,\infty)$, see \cite{haasetransference11,Vitse05}. Obviously, the (inverse) Fourier-Laplace transform of $f\ast\hat{h_{k}}$ has support in $[2^{k-1},2^{k+1}]$, hence, $f\ast\hat{h_{k}}\in H^{\infty}[2^{k-1},2^{k+1}]$. 
		Therefore, it follows directly from Theorem \ref{cor1} that for $f\in B_{\infty,1}^{0}$
			\begin{equation}\label{eq1:secBesov}
				\|(f\ast\hat{h_{k}})(A)\| \leq c M (\log(M)+1) \cdot4\cdot \|f\ast\hat{h_{k}}\|_{\infty} ,
			\end{equation}	
		where $c$ is an absolute consant and $M=M(A,\frac{\pi}{2})$. The following Theorem is a slight improvement of Theorem 1.7 in \cite{Vitse05}, see also \cite[Cor.~5.5]{haasetransference11}.
		\begin{theorem}\label{thm:Besov}
		Let $A\in\Sect(\omega)$ on the Banach space $X$ with $\omega<\frac{\pi}{2}$. Let $M=M(A,\frac{\pi}{2})$. Then,
		\begin{equation*}
			\|f(A)\| \leq c M(\log(M)+1)\|f\|_{*B},
		\end{equation*}
		for all $f\in B_{\infty,1}^{0}$, where $c>0$ is an absolute constant. Thus, the $B_{\infty,1}^{0}$-calculus is bounded.
		\end{theorem}
		\begin{proof}
		It is easy to see that for $g\in H^{\infty}[\veps,\sigma]$ with $0<\veps<\sigma<\infty$, 
		\begin{equation}\label{eq2:secBesov}
			g(z)= \sum_{k\in\mathbb{Z}} (\hat{h_{k}}\ast g)(z),\quad z\in\Cp
		\end{equation}
		because the inverse Fourier transform of $g$ has compact support. \newline Let $f\in B_{\infty,1}^{0}$. Since $\cup _{0<\veps<\sigma}H^{\infty}[\veps,\sigma]$ is dense in $B_{\infty,1}^{0}$, see \cite{Vitse05}, we find a sequence $g_{n}\in\Hinf[\frac{1}{n},n]$ such that $g_{n}\to (f-f(\infty))$ in $B_{\infty,1}^{0}$ as $n\to\infty$. Thus, $g_{n}\to f-f(\infty)$ in $\|\cdot\|_{\infty}$ and $\|\cdot\|_{\ast B}$. Therefore, by (\ref{eq2:secBesov}) and the fact that $\hat{h_{k}}\ast(f-f(\infty))=\hat{h_{k}}\ast f$ we have that
		\begin{equation}\label{eq3:secBesov}
			f(z)=f(\infty)+\sum_{k\in\mathbb{Z}}(\hat{h_{k}}\ast f)(z), \quad z\in\Cp.
		\end{equation}
		Since $\|\sum_{|k|\leq N} (\hat{h_{k}}\ast f)\|_{\infty}\leq \|f\|_{\ast B}$ for $N\in\N$, the Convergence Lemma, \cite[Prop.~5.1.4]{haasesectorial}, implies 
		\begin{equation*}
			f(A)=f(\infty)+\sum_{k\in\mathbb{Z}}(\hat{h_{k}}\ast f)(A)
		\end{equation*}
		and the assertion follows from (\ref{eq1:secBesov}).
		\end{proof}
		\begin{remark}
		\begin{enumerate}
			\item In \cite[Thm.~1.7]{Vitse05} Vitse already showed that the $B_{\infty,1}^{0}$-calculus is bounded where the bound of the calculus was estimated by $31 M^{3}$. Like in our proof, she derived the result from an $\Hinf$-calculus estimate for $\Hinf[\veps,\sigma]$. 
			\item In \cite{haasetransference11} Haase showed that for (polynomially) bounded semigroups on Hilbert spaces, one can consider more general homogenous Besov spaces $B_{\infty,1}^{s}$, $s\geq0$. $B_{\infty,1}^{s}$ consists of functions $f$, holomorphic on $\Cp$, and such that $\lim_{z\to\infty}f(z)$ exists and 
			\begin{equation*}
				\|f\|_{\ast B^{s}}:=|f(\infty)| + \sum_{k<0} \|\hat{h_{k}}\ast f\|_{\infty} + \sum_{k\geq0} 2^{ks}\|\hat{h_{k}}\ast f\|_{\infty}<\infty. 
			\end{equation*}
			It is easy to see that Theorem \ref{thm:Besov} holds for $B_{\infty,1}^{s}$ with the analogous proof as for $B_{\infty,1}^{0}$.
		\end{enumerate}
		
		\end{remark}
			
	\subsection{Final remarks and outlook}
		Let us conclude by mentioning the well-known relation between analytic semigroup generators and Tadmor--Ritt operators, see e.g.\ \cite{haasesectorial,Vitse2005b,Vitse05}. A bounded operator $T$ is called \textit{Tadmor--Ritt} if its spectrum lies in the closed unit disc and its resolvent satisfies that
		\begin{equation*}
				C(T):=\sup_{|z|>1}\|(z-1)R(z,T)\|<\infty,
		\end{equation*}
		see \cite{Ritt53,TadmorLAA}. Such operators are of interest in the study of stability of numerical schemes. 
		Moreover, they can be seen as the discrete counterpart of sectorial operators. 
		In \cite{Vitse04,Vitse2005b}, Vitse  discussed $H^{\infty}$- and Besov space functional calculi for Tadmor--Ritt operators with similar ideas as in the continuous case, \cite{Vitse05}. It seems natural to use discrete versions of the techniques used in this paper to improve these results. Such results were recently obtained by the author, \cite{SchwenningerTR}.\\
		We point out that in Theorems \ref{thm1} and \ref{thm:expstab} the operator $A$ need not be densely defined. 
		Thus, in the view of analytic semigroups, $e_{t}(A)$ need not be strongly continuous at $0$, see \cite[Sec.~3.3]{haasesectorial}. 
		
		Looking back to Propositions \ref{prop1} and \ref{prop3} which served as a starting point to study $\|(fe_{\veps})\|$ to quantify the (un)boundedness, we can ask ourselves which other functions $g_{\veps}$ with $g_{\veps}\to 1$ as $\veps\to0$ can be studied in order to characterize a bounded calculus. For example, one could consider $g_{\veps}(z)=z^{\veps}e^{-\veps z}$ which yields that $fg_{\veps}\in\HinfO(\Sigma_{\delta})$ for $f\in\Hinf(\Sigma_{\delta})$.\\
		Another question is how Theorem \ref{thm:sqfctestlog} generalizes to general Banach spaces. 
		As Theorem \ref{thm:McIntosh} is not true on general Banach spaces, one has to use generalized square function estimates to characterize bounded $\Hinf$-calculus then, see e.g.\ \cite{cowlingdoustmcintoshyagi,haakthesis,KunstmannWeis04}. This is subject to future work.
 
 	\section*{Acknowledgements}
	The author has been supported by the Netherlands Organisation for Scientific Research (NWO), grant no.\ 613.001.004, while he was at the University of Twente, where most of the present work was carried out.
	He is very grateful to Hans Zwart for numerous discussions, his comments on the manuscript and many helpful suggestions. 
	The author would also like to thank Markus Haase and Jan Rozendaal for fruitful conversations on functional calculus. 
	He is very thankful to Dmitry Yakubovich for bringing the relations to the paper \cite{Yakubovich2011} to his attention. \newline Finally, he wants to express his gratitude to the anonymous referee for his very careful reading and his suggestions which helped a lot to improve the presentation of the paper.
	
 	\begin{appendix}
	\section{Some results about certain Schauder bases}
	\begin{applemma}[Growth Lemma]\label{le:growthlemma}
	Let $b>1$ and $\gamma_{0}\in(-1,0)$.
	\begin{enumerate}[label={(\roman*)}]%,ref={\theapplemma(\roman*)}]
	\item There exist $c_{\gamma_{0},b}$, $C_{\gamma_{0},b}>0$ such that\label{itGL1} for $0<\veps<\frac{1}{2b}$ and $\gamma\in(\gamma_{0},0)$,
		\begin{align}
		c_{\gamma_{0},b} \log\left(\tfrac{1}{\veps}\right)^{1+\gamma}&\leq\sum_{n=1}^{\infty}n^{\gamma}e^{-b^{n}\varepsilon}
		\leq C_{\gamma_{0},b} \log\left(\tfrac{1}{\veps}\right)^{1+\gamma}. \label{eq:growthlemma2}
	\end{align}
	%where $F_{\gamma}(\veps,b)=\frac{\log(1/\veps)^{1+\gamma}-\log(b)^{1+\gamma}}{\log(b)^{1+\gamma}(1+\gamma)}$ 
%	\begin{align}
%		e^{-1}F_{\gamma}(\veps,b)&\leq\sum_{n=1}^{\infty}n^{\gamma}e^{-b^{n}\varepsilon}
%		\leq F_{\gamma}(\veps,b)+1+\tfrac{{\rm Ei}(1)}{\log(b)},\label{eq:growthlemma2}
%	\end{align}
%	where $F_{\gamma}(\veps,b)=\frac{\log(1/\veps)^{1+\gamma}-\log(b)^{1+\gamma}}{\log(b)^{1+\gamma}(1+\gamma)}$ 

%	\begin{equation*}
%		F_{\gamma}(\veps,b)=\left\{\begin{array}{cc}\frac{\log(1/\veps)^{1+\gamma}-\log(b)^{1+\gamma}}{\log(b)^{1+\gamma}(1+\gamma)},& \gamma\neq -1,\\ \log\log(1/\veps)-\log\log(b),& \gamma=-1.\end{array}\right.
%	\end{equation*}
%		and for all $\gamma_{0}\in(-1,0)$ there exists $C_{\gamma_{0},b}>0$ such that
%		\begin{equation}\label{eq:growthlemma3}
%		 	\forall \gamma\in(\gamma_{0},0):\quad F_{\gamma}(\veps,b)\geq C_{\gamma_{0},b} \log\left(\tfrac{1}{\veps}\right)^{1+\gamma}.
%		\end{equation}
	\item\label{itGL2} For all $\veps>0$, $\sum_{n=1}^{\infty}e^{-b^{n}\varepsilon}\leq\frac{{\rm Ei}(\veps)}{\log(b)}.$
% \begin{equation*}\sum_{n=1}^{\infty}e^{-b^{n}\varepsilon}\leq\frac{{\rm Ei}(\veps)}{\log(b)}.
%	\end{equation*}
\end{enumerate}
	\end{applemma}
	
	\begin{proof}
		We estimate $\int_{1}^{\infty}x^{\gamma}e^{-b^{x}\varepsilon}dx$. Substitute $y=b^{x}\varepsilon$, thus, $x=\frac{\log(y/\varepsilon)}{\log(b)}$,
	 \begin{align*}
	 \int_{1}^{\infty}x^{\gamma}e^{-b^{x}\varepsilon}\ dx={}&\tfrac{1}{\log(b)^{1+\gamma}}\int_{\varepsilon b}^{\infty}\log\left(\frac{y}{\varepsilon}\right)^{\gamma}\frac{e^{-y}}{y} \ dy\\
	 ={}&\tfrac{1}{\log(b)^{1+\gamma}}~\Big(\int_{\varepsilon b}^{1}\log\left(\frac{y}{\varepsilon}\right)^{\gamma}\frac{e^{-y}}{y} \ dy+\underbrace{\int_{1}^{\infty}\log\left(\frac{y}{\varepsilon}\right)^{\gamma}\frac{e^{-y}}{y} \ dy}_{\leq \log(\frac{1}{\veps})^{\gamma} {\rm Ei}(1)< \log(b)^{\gamma} {\rm Ei}(1)}\Big).
	\end{align*} 
	Because $e^{-1}\leq e^{-y}\leq1$ for $y\in (\varepsilon b,1)$ and
since the primitive of $\frac{\log(y/\varepsilon)^{\gamma}}{y}$ is $\frac{(\log(y/\varepsilon))^{1+\gamma}}{1+\gamma}$,
%\begin{equation*}
%\left\{\begin{array}{cc}\frac{(\log(y/\varepsilon))^{1+\gamma}}{1+\gamma},&\gamma\neq-1,\\\log\log(y/\veps),&\gamma=-1,\end{array}\right.
%\end{equation*}
 we obtain
 \begin{equation}\label{eq2:LeA1}
e^{-1}\tfrac{\log(1/\veps)^{1+\gamma}-\log(b)^{1+\gamma}}{\log(b)^{1+\gamma}(1+\gamma)}\leq \int_{1}^{\infty}x^{\gamma}e^{-b^{x}\varepsilon}\ dx \leq 
		\tfrac{\log(1/\veps)^{1+\gamma}-\log(b)^{1+\gamma}}{\log(b)^{1+\gamma}(1+\gamma)}+\tfrac{{\rm Ei}(1)}{\log(b)}.
			\end{equation}	
  Next we use that for the decreasing, integrable function $f:[1,\infty)\rightarrow \Rp$, $x\mapsto x^{\gamma}e^{-b^{x}\veps}$ holds that
	\begin{equation*}%\label{integralcriterion}
	\int_{1}^{\infty}f(x)\ dx\leq \sum_{n=1}^{\infty}f(n)\leq f(1)+\int_{1}^{\infty} f(x)\ dx.
	\end{equation*}
	(\ref{eq:growthlemma2}) follows by estimating the left and the right-hand side term in \eqref{eq2:LeA1}.
	Finally, \ref{itGL2} follows by
	\begin{align*}
		 \sum_{n=1}^{\infty}e^{-b^{n}\varepsilon}\leq {}& \int_{0}^{\infty}e^{-b^{x}\varepsilon}\ dx
				=\frac{1}{\log(b)}\int_{\varepsilon}^{\infty}\frac{e^{-y}}{y} \ dy
					=\frac{{\rm Ei}(\veps)}{\log(b)}.
	\end{align*}
	\end{proof}
	
		\begin{applemma}\label{lem:behaviorcoeff}
	There exist $c_{1},C_{1},C_{2}>0$ such that the following holds for all $n\in\N$.
	\begin{align}\label{lem:behaviorcoeffeq0}
	\forall\alpha\in(-1,1):\quad c_{n,\alpha}:=\int_{-\frac{\pi}{2}}^{\frac{\pi}{2}} |t|^{\alpha} e^{int} \ dt =C_{1,\alpha}n^{-1-\alpha} + B_{n,\alpha}\in\R,
	\end{align}
	where $C_{1,\alpha}=-2\sin\left(\alpha\frac{\pi}{2}\right)\Gamma(\alpha+1)$, $|B_{n,\alpha}|\leq C_{2}n^{-1}$. Moreover, 
	 %Moreover, $c_{n,\alpha}\in\R$ and for $\alpha\in(-1,0]$,  
	\begin{equation}\label{lem:behcoeffeq2}
	\forall \alpha\in\left(-1,-\tfrac{5}{12}\right]:\quad \tfrac{c_{1}}{1+\alpha}~n^{-1-\alpha}\leq c_{n,\alpha} \leq \tfrac{C_{1}}{1+\alpha}~n^{-1-\alpha}.
	%d_{3,\alpha} n^{-1-\alpha}\leq c_{n,\alpha} \leq d_{1,\alpha}n^{-1-\alpha},\quad n\in\N,
	\end{equation}
	%with $d_{k,\alpha}=2\int_{0}^{\frac{k\pi}{2}}t^{\alpha}\cos t\ dt\sim \frac{1}{1+\alpha}$, for $k\in\left\{1,3\right\}$. If $\alpha\in(-1,-\frac{5}{12}]$, then $c_{n,\alpha}\geq d_{3,\alpha}>0$, $n\in\mathbb{N}$.
	
	\end{applemma}
	\begin{proof}
	By
	%\begin{equation*}
$		c_{n,\alpha}=\int_{-\frac{\pi}{2}}^{\frac{\pi}{2}}|t|^{\alpha}e^{int}dt=2\Re\int_{0}^{\frac{\pi}{2}}t^{\alpha}e^{int}dt$,
%	\end{equation*}
	 it is clear that $c_{n,\alpha}$ is real and we can consider
	\begin{equation}\label{lem:behaviorcoeffeq1}
	\int_{0}^{\frac{\pi}{2}}t^{\alpha}e^{int} dt=n^{-1-\alpha}\int_{0}^{n\frac{\pi}{2}}t^{\alpha}e^{it}dt.
	\end{equation}
	Consider the contour consisting of the lines segments $[\veps,n\frac{\pi}{2}]$ and $i[\veps,n\frac{\pi}{2}]$ connected via quarter circles with radii $n\frac{\pi}{2}$ and $\veps$ respectively, orientated counterclockwise. Then, since $h(z)=z^{\alpha}e^{iz}$ is holomorphic on $\C\setminus\left\{z\in\R:z\leq0\right\}$,
	\begin{equation}\label{lem:behaviorcoeffeq11}
	\int_{\veps}^{n\frac{\pi}{2}}h(t)dt=\int_{\veps}^{n\frac{\pi}{2}}h(it)idt-i\int_{0}^{\frac{\pi}{2}}(n\frac{\pi}{2}e^{i\theta})^{\alpha+1}e^{in\frac{\pi}{2}e^{i\theta}}d\theta+i\int_{0}^{\frac{\pi}{2}}(\veps e^{i\theta})^{\alpha+1}e^{i\veps e^{i\theta}}d\theta.
	\end{equation}
	The last two integrals can both be estimated using the fact that $|e^{ir e^{i\theta}}|=e^{-r\sin\theta}\leq e^{-r\frac{2\theta}{\pi}}$ for $\theta\in[0,\frac{\pi}{2}]$, $r>0$. This yields
	\begin{equation*}
	\left|\int_{0}^{\frac{\pi}{2}} (re^{i\theta})^{\alpha+1}e^{ire^{i\theta}}d\theta\right|\leq \frac{\pi}{2} r^{\alpha}(1-e^{-r}).
	\end{equation*}
	Therefore, the integral for $r=\veps$ goes to zero as $\veps\to0^{+}$ because $\alpha>-1$. The integral for $r=n\frac{\pi}{2}$ can be estimated by $\left(\frac{\pi}{2}\right)^{\alpha+1}n^{\alpha}$. 
	It remains to consider
	\begin{align*}
		\lim_{\veps\to0^{+}}i\int_{\veps}^{n\frac{\pi}{2}}h(it)dt=i\int_{0}^{n\frac{\pi}{2}}h(it)dt={}&e^{i(\alpha+1)\frac{\pi}{2}}\int_{0}^{n\frac{\pi}{2}}t^{\alpha}e^{-t}dt
					=e^{i(\alpha+1)\frac{\pi}{2}}\left[\Gamma(\alpha+1)-\int_{n\frac{\pi}{2}}^{\infty}t^{\alpha}e^{-t}dt\right].
	\end{align*}
	It is easily seen that there exists a constant $C$ such that $\int_{n}^{\infty}t^{\alpha}e^{-t}dt\leq C n^{\alpha}e^{-n}$ for all $\alpha\in(-1,1)$.  
	Altogether we get by (\ref{lem:behaviorcoeffeq1}) and the estimates for the terms in (\ref{lem:behaviorcoeffeq11}) that 
	\begin{equation*}
	\int_{0}^{\frac{\pi}{2}}t^{\alpha}e^{int} dt =e^{i(\alpha+1)\frac{\pi}{2}}\Gamma(\alpha+1)n^{-1-\alpha}+B_{n,\alpha},
	\end{equation*}
	with $|B_{n,\alpha}|\leq \frac{1}{n}\left[\left(\frac{\pi}{2}\right)^{\alpha+1}+Ce^{-n}\right]$.
		This yields (\ref{lem:behaviorcoeffeq0}).
		
		To show (\ref{lem:behcoeffeq2}) for $\alpha\in(-1,-\frac{5}{12}]$, note that by (\ref{lem:behaviorcoeffeq1}),
		\begin{equation*}
		c_{n,\alpha}=n^{-1-\alpha}2\int_{0}^{\frac{n\pi}{2}}t^{\alpha}\cos(t) dt.		\end{equation*}
		We define $d_{n,\alpha}=2\int_{0}^{\frac{n\pi}{2}}t^{\alpha}\cos (t) dt$ and show that $d_{3,\alpha}\leq d_{n,\alpha}\leq d_{1,\alpha}$ for $n\in\N$.
		Since $t\mapsto t^{\alpha}$ is positive and decreasing on $(0,\infty)$ it follows, by the periodicity of $\cos$ that for all $m\in\N_{0}$,
		\begin{enumerate}
			\item $d_{4m+1,\alpha}>d_{4m+2,\alpha}>d_{4m+3,\alpha}$, since $\cos (\frac{t\pi}{2})<0$ on $\left((4m+1),(4m+3)\right)$,  
		  \item $d_{4m+3,\alpha}<d_{4m+4,\alpha}<d_{4m+5,\alpha}$, since $\cos(\frac{t\pi}{2})>0$ on $\left((4m+3),(4m+5)\right)$, 
			\item $d_{4m+5,\alpha}<d_{4m+1,\alpha}$ and $d_{4m+3,\alpha}<d_{4(m+1)+3,\alpha}$
, since $t\mapsto t^{\alpha}$ is decreasing. 		\end{enumerate}
		 Inductively, this shows that $\max_{n}d_{n,\alpha}=d_{1,\alpha}$ and $\min_{n}d_{n,\alpha}=d_{3,\alpha}$. \newline
		Finally, we check that $d_{3,\alpha}>0$ if $\alpha\in(-1,-\frac{5}{12}]$,
		\begin{align*}
	d_{3,\alpha}={}&\int_{0}^{\frac{3\pi}{2}}t^{\alpha}\cos(t)\ dt\geq\int_{[0,1]\cup[\frac{\pi}{2},\frac{3\pi}{2}]}t^{\alpha}\cos(t)\ dt+\left(\tfrac{\pi}{2}\right)^{\alpha}\int_{1}^{\frac{\pi}{2}}\cos(t)dt%+\int_{\frac{\pi}{2}}^{\frac{3\pi}{2}}t^{\alpha}\cos(t)dt\\
	\\
	\geq{}&\cos(t_{0})\int_{0}^{t_{0}}t^{\alpha}dt+\stackrel{>0}{\int_{[t_{0},1]\cup[\frac{\pi}{2},\frac{3\pi}{2}]}t^{-\frac{5}{12}}\cos(t)\ dt+\tfrac{2}{\pi}(1-\sin(1))}\stackrel{(*)}{\geq}\tfrac{\cos(t_{0})~ t_{0}^{1+\alpha}}{1+\alpha} \geq\frac{c_{1}}{1+\alpha},
	 \end{align*}
	 where $(*)$  follows for some $t_{0}\in(0,1)$ such that  
	% \begin{equation*}
	$ \int_{[t_{0},1]\cup[\frac{\pi}{2},\frac{3\pi}{2}]}t^{-\frac{5}{12}}\cos(t)\ dt+\tfrac{2}{\pi}(1-\sin(1))> 0$.
	% \end{equation*}
	   The existence of such $t_{0}$ can be shown using Fresnel integrals. Clearly, $d_{1,\alpha}\leq \frac{C_{1}}{1+\alpha}$.
		\end{proof}

	\begin{applemma}\label{le:settingexample}
	Let $X=L^{2}(-\pi,\pi)$. Then there exist $c_{i}$, $C_{i}>0$, $i\in\{1,..,3\}$ such that for all $\beta\in(\frac{1}{4},\frac{1}{2})$, $w_{\beta}$ and $\left\{\Phi_{n}\right\}_{n\in\N}$ as in Definition \ref{def:L2basis},
	the following assertions hold.
	
	\begin{enumerate}[label={(\roman*)},ref={\theapplemma~(\roman*)}]
		\item\label{le:settingexampleit1} $\left\{\Phi_{n}\right\}_{n\in\N}$ forms a bounded Schauder basis of $X$ with $\frac{c_{1}}{1-2\beta}\leq\kappa_{\Phi}\leq \frac{C_{1}}{ 1-2\beta}$,  (see  (\ref{def:basisconstants}) for $\kappa_{\Phi}$).	
		\item\label{le:settingexampleit3} The family $\left\{\Phi_{n}^{*}\right\}_{n\in\N}\subset X$ given by
	%\begin{equation*}
	$\Phi_{2k}^{*}(t)=\tfrac{1}{2\pi w_{\beta}(t)}e^{ikt}$,  $\Phi_{2k+1}^{*}(t)=\tfrac{1}{2\pi w_{\beta}(t)}e^{-ikt}$, 
	%\end{equation*}
	satisfies $\langle \Phi_{n}^{*},\Phi_{m}\rangle_{L^{2}} =\delta_{nm}$ 
	and forms a Schauder basis with $\frac{c_{2}}{1-2\beta}\leq\kappa_{\Phi^{*}}\leq \frac{C_{2}}{1-2\beta}$. 
	%Here, $w_{\beta}$ is defined as in Definition \ref{def:L2basis}.	
		\item\label{le:settingexampleit4} The coefficients of $x(t)= |t|^{-\beta}\mathbb{1}_{(0,\frac{\pi}{2})}(|t|)$, $x=\sum_{n}x_{n}\Phi_{n}$ are positive and satisfy
			\begin{equation}\label{eq:alphaknest}
		 c_{3} \tfrac{k^{-1+2\beta}}{1-2\beta} \leq x_{2k}= x_{2k+1}\leq C_{3} \tfrac{k^{-1+2\beta}}{1-2\beta} ,\qquad k\in\N\cup\{0\}.
			\end{equation}
	For the coefficients of $y(t)=(\pi-|t|)^{-\beta}\mathbb{1}_{(\frac{\pi}{2},\pi)}(|t|)$, $y=\sum_{n}y_{n}\Phi_{n}^{*}$, we have that 
	\begin{equation}\label{coeffyk}
	y_{2k}=(-1)^{k}2\pi\cdot x_{2k},\quad y_{2k+1}=(-1)^{k} 2\pi \cdot x_{2k+1}, \qquad k\in\N\cup\{0\}.
	\end{equation}
	\end{enumerate}
	\end{applemma}
	\begin{proof}
	%Lemma \ref{le:settingexampleit1}-\ref{le:settingexampleit3} follow from \cite[Example II.11.2, p.~351]{SingerBases}. \newline
	The fact that $\left\{\Phi_{n}\right\}_{n\in\N}$ and $\left\{\Phi_{n}^{*}\right\}_{n\in\N}$ form bounded Schauder bases can for instance be found in \cite[Lem.~4.1 and Ex.~4.4]{EisnerZwart06}. In the proof of \cite[Lem.~4.1]{EisnerZwart06}, one can find a known method to derive the basis constants $\kappa_{\Phi}$, $\kappa_{\Phi^{*}}$ from the bound of the Hilbert-transform acting on weighted $L^{2}$ spaces (with $A_{2}$-weights), see also \cite{Petermichl07}. It is easy to see that $\langle \Phi_{n}^{*},\Phi_{m}\rangle_{L^{2}} =\delta_{nm}$.
	\smallskip
	
	To see \ref{le:settingexampleit4} we point out that for all $x=\sum_{n}x_{n}\Phi_{n}\in X$ there holds
		\begin{equation}\label{lem:basiseq1}
		x_{n}=\langle x, \Phi_{n}^{*}\rangle _{L^{2}}, \qquad n\in\N.
		\end{equation}
		Thus, for $x=(t\mapsto |t|^{-\beta}\mathbb{1}_{(0,\frac{\pi}{2})}(|t|))$, $k\in\N$,
	\begin{equation}\label{lem:basiseq2}
	x_{2k}=\frac{1}{2\pi}\int_{-\pi/2}^{\pi/2}|t|^{-2\beta}e^{-ik t} dt=\frac{c_{k,-2\beta}}{2\pi}, \quad x_{2k+1}=\frac{1}{2\pi}\int_{-\pi/2}^{\pi/2}|t|^{-2\beta}e^{ik t}dt=x_{2k},
	\end{equation}
	where $c_{k,-2\beta}$ are the coefficients from Lemma \ref{lem:behaviorcoeff}. Moreover, since $-2\beta\in(-1,-\frac{1}{2})$, (\ref{eq:alphaknest}) follows by \eqref{lem:behcoeffeq2}. It is easy to see that $\frac{c_{3}}{1-2\beta}\leq x_{1}\leq \frac{C_{3}}{1-2\beta}$.
	The assertion for $y$ follows similarly.
	\end{proof}
	\begin{applemma}\label{le:settingexampleit6} Let $X=L^{2}=L^{2}(-\pi,\pi)$. There exist $c_{i}$, $C_{i}>0$, $i\in\{1,..,6\}$ such that for all $\beta\in(\frac{5}{12},\frac{1}{2})$ and $\left\{\Psi_{n}\right\}_{n\in\N}\subset X$  defined by
				\begin{equation*}
				\Psi_{2k}(t)=|t|^{\beta}e^{ikt}, \quad \Psi_{2k+1}(t)=|t|^{\beta}e^{-ikt},\ k\in\N\cup\{0\},
				\end{equation*}
				the following assertions hold.
				\begin{enumerate}[label=(\roman*)]
				\item\label{le:A4i} $\{\Psi_{n}\}_{n\in\N}$ is a bounded Schauder basis with $c_{1}(1-2\beta)^{-1/2}\leq\kappa_{\Psi}\leq C_{1} (1-2\beta)^{-1/2}$.
				\item\label{le:A4ii} The family $\left\{\Psi_{n}^{*}\right\}_{n\in\N}\subset X$ given by
	%\begin{equation*}
	$\Psi_{2k}^{*}(t)=\tfrac{1}{2\pi}|t|^{-\beta}e^{ikt}$, $\Psi_{2k+1}^{*}(t)=\tfrac{1}{2\pi}|t|^{-\beta}e^{-ikt},$
	%\end{equation*}
	satisfies $\langle \Psi_{n}^{*},\Psi_{m}\rangle_{L^{2}} =\delta_{nm}$ 
	and forms a Schauder basis with $c_{2}(1-2\beta)^{-1/2}\leq\kappa_{\Psi^{*}}\leq C_{2}(1-2\beta)^{-1/2}$.
	\item\label{le:A4iii}  $\left\{\Psi_{n}^{*}\right\}_{n\in\N}\subset X$ is Besselian, i.e.
				%\begin{equation}\label{eq:Besselian}
				$\forall y=\sum_{n\in\N}y_{n}\Psi_{n}^{*}\in X \ \Rightarrow \ (y_{n})\in\ell^{2}(\N).$
				%\end{equation}
	\item \label{le:A4iv}	The coefficients of $x(t)= |t|^{-\beta}\mathbb{1}_{(0,\frac{\pi}{2})}(|t|)$, $x=\sum_{n}x_{n}\Psi_{n}$ are positive and satisfy
			\begin{equation}\label{eq:alphaknest2}
		 c_{3} \tfrac{k^{-1+2\beta}}{1-2\beta} \leq x_{2k}= x_{2k+1}\leq C_{3} \tfrac{k^{-1+2\beta}}{1-2\beta} ,\qquad k\in\N\cup\{0\}.
			\end{equation}
	For the coefficients of $y(t)=|t|^{-\beta}(\pi-|t|)^{-\beta}\mathbb{1}_{(\frac{\pi}{2},\pi)}(|t|)$, $y=\sum_{n}y_{n}\Psi_{n}^{*}$ we have that 
	\begin{equation}\label{coeffyk2}
	y_{2k}=y_{2k+1}=(-1)^{k}\:\frac{c_{k,-\beta}}{2\pi} \quad \text{and}\quad c_{4}k^{-1+\beta}\leq|y_{2k}| \leq C_{4}k^{-1+\beta}, \qquad k\in\N\cup\{0\}.
	\end{equation}	
	\item	\label{le:A4v}		For $x=\sum_{n\in\N}x_{n}\Psi_{n}\in X$, we have that $\left\{x_{n}\right\}\in\ell^{r}$ for $r>\frac{2}{1-2\beta}$ and
				\begin{equation}\label{le:basiseqYoung}
				\|(x_{n})\|_{r} \leq C_{5} \|x\| \cdot \|n^{-1+\beta}\|_{q},\qquad \tfrac{1}{q}=\tfrac{1}{2}+\tfrac{1}{r}.
				\end{equation}
				Furthermore, 
				\begin{equation}\label{eq23:Thm42}
			\|(e^{-2^{n}\veps}x_{n})\|_{2}\leq  C_{6}\:\|(n^{\beta-1})\|_{\frac{3-2\beta}{4}}\cdot {\rm Ei}(\veps)^{\frac{1+2\beta}{4}}\|x\|.	
				\end{equation} 
				%with $K_{\beta}=\|(n^{\beta-1})\|_{\frac{3-2\beta}{4}}$. The constants involved in $\lesssim$ are absolute and independent of $x,\beta,\veps$.
				\end{enumerate}
				\end{applemma}
				\begin{proof}
				Proofs for \ref{le:A4i}--\ref{le:A4iii} can be found in \cite[Ex.~II.11.2, p.~351]{SingerBases}. Since the value of  $\kappa_{\Psi}$ is not obvious there, we refer to \cite[Lem.~4.1]{EisnerZwart06} how to derive $\kappa_{\Psi}$, see also the proof of Lemma \ref{le:settingexample}.
				\smallskip
				
				\ref{le:A4iv}: Since $\Psi_{n}(t)=\Phi_{n}(t)$ for $t\in(-\frac{\pi}{2},\frac{\pi}{2})$, with $\Phi_{n}$ from Definition \ref{def:L2basis}, it follows that $x$ has same coefficients $x_{n}$ with respect to $\left\{\Psi_{n}\right\}$ as for the basis $\left\{\Phi_{n}\right\}$. Thus, \eqref{eq:alphaknest2} holds by Lemma \ref{le:settingexampleit4}. 
		The coefficients of $y=\sum_{n}y_{n}\Psi_{n}^{*}$ are derived by using $\langle\Psi_{n},\Psi_{m}^{*}\rangle_{L^{2}}=\delta_{nm}$, $k\in\N$,
	\begin{equation*}
		y_{2k}=\langle y,\Psi_{2k}\rangle_{L^{2}}=\frac{1}{2\pi}\int_{\frac{\pi}{2}<|t|<\pi}(\pi-|t|)^{-\beta}e^{ikt}dt=\frac{(-1)^{k}}{2\pi}\int_{-\frac{\pi}{2}}^{\frac{\pi}{2}}|t|^{-\beta}e^{-ikt}dt=(-1)^{k}\:\frac{c_{k,-\beta}}{2\pi}.
	\end{equation*}
	Furthermore, $y_{2k+1}=y_{2k}$ and
	Lemma \ref{lem:behaviorcoeff} yields the estimate in \eqref{coeffyk2} since $1-\beta\in(\frac{1}{2},\frac{7}{12})$.	
				\smallskip
				
				To show \ref{le:A4v}, let $w_{\beta}(t)=|t|^{\beta}$ on $(-\pi,\pi)$. Since $\{e^{int}\}_{n\in\mathbb{Z}}$ is an orthogonal basis of $L^2$, it follows that for $x=\sum_{n\in\N}x_{n}\Psi_{n}\in X$,
	\begin{equation*}
	x_{2k}=\frac{1}{2\pi}\langle x w_{\beta}^{-1},e^{ik\cdot}\rangle_{L^{2}} =\mathcal{F}(xw_{\beta}^{-1})[k],
	\end{equation*}
	where $\mathcal{F}$ denotes the discrete Fourier transform. Thus,
	\begin{equation}\label{lem:basiseq3}
	x_{2k}=\left(\mathcal{F}(x)\ast \mathcal{F}(w_{\beta}^{-1})\right)[k].
	\end{equation}
	By $x\in L^{2}$, $\left\{\mathcal{F}(x)[n]\right\}\in\ell^{2}$. From \cite[Proof of Thm.~2.4, p.861]{Haak2012}  (see also  Lemma \ref{lem:behaviorcoeff}) we have 
	\begin{equation*}
		\int_{-\pi}^{\pi}|t|^{\gamma-1}e^{-int}\ dt =2 n^{-\gamma}\cos(\gamma\tfrac{\pi}{2})\Gamma(\gamma)+B_{n,\gamma}, 
	\end{equation*}
	for $\gamma>0$ and with $|B_{n,\gamma}|\leq \tfrac{C}{n}$ for some absolute constant $C$.
	Thus, with $\gamma=1-\beta\in(\frac{1}{2},\frac{7}{12})$, $\mathcal{F}(w_{\beta}^{-1})[n]\in \ell^{q}$ with $q>q_{0}:=\frac{1}{1-\beta}$,
	\begin{equation*}
	\|\mathcal{F}(w_{\beta}^{-1})[n]\|_{q}\leq \|(n^{-1+\beta})\|_{q}~\max\nolimits_{\gamma\in(\frac{1}{2},\frac{7}{12})}|\cos(\gamma\tfrac{\pi}{2})\Gamma(\gamma)|+\|(n^{-\frac{12}{7}})\|_{\frac{7}{12}}\leq C_{5}\|(n^{-1+\beta})\|_{q}.
	\end{equation*}
	We use Young's inequality with $\frac{1}{2}+\frac{1}{q}=1+\frac{1}{r}$ and $q\in(q_{0},2)$ to estimate the right-hand-side of (\ref{lem:basiseq3}). Hence, $\left\{x_{2k}\right\}\in\ell^{r}$ for $r>r_{0}:=\frac{2}{1-2\beta}$. Analogously, $\left\{x_{2k+1}\right\}\in\ell^{r}$. Eq.\ (\ref{le:basiseqYoung}) then follows since the discrete Fourier transform is isometric from $L^{2}$ to $\ell^{2}$. \smallskip\\
	To show (\ref{eq23:Thm42}), we use  H\"older's inequality and (\ref{le:basiseqYoung}), 
		\begin{align}
		\|(e^{-2^{n-1}\veps}x_{n})\|_{2}^{2}={}&\|(e^{-2^{n}\veps}|x_{n}|^{2})\|_{1}
		\leq{}\|(e^{-2^{n}\veps})\|_{r_{0}'}\: \|(x_{n})\|_{2r_{0}}^{2}
		\leq{} C_{5}\|(e^{-2^{n}\veps})\|_{r_{0}'} \: \|(n^{-1+\beta})\|_{q}^{2}\:\|x\|^{2}\label{eq13:Thm42}
		\end{align}
		for $r_{0}'=(1-\frac{1}{r_{0}})^{-1}=\frac{2}{1+2\beta}$ and $\frac{1}{q}=\frac{1}{2}+\frac{1}{2r_{0}}=\frac{3-2\beta}{4}$. 
		By Lemma \ref{le:growthlemma}~\ref{itGL1},
		\begin{equation}\label{eq23:Thm422}
			\|(e^{-2^{n}\veps})\|_{r_{0}'=\frac{2}{1+2\beta}}\leq \left(\tfrac{1}{\log(2)}{\rm Ei}(r_{0}'\veps)\right)^{\frac{1+2\beta}{2}} \stackrel{(\ref{eq:expint})}{\leq} \log(2)~{\rm Ei}(\veps)^{\frac{1+2\beta}{2}},		
		\end{equation} 
		where we used that $r_{0}'>1$.	Thus, (\ref{eq13:Thm42}) shows (\ref{eq23:Thm42}).
				\end{proof}

\end{appendix}

%\small
%style abbrv to abbriviate first names
%amsalpha for [Name89]
 %  \bibliographystyle{abbrv} 
 
%\bibliography{../../../refsBigProject}

\begin{thebibliography}{10}

\bibitem{AbramowitzExpInt}
M.~Abramowitz and I.~A. Stegun.
\newblock {\em Handbook of mathematical functions with formulas, graphs, and
  mathematical tables}, volume~55 of {\em National Bureau of Standards Applied
  Mathematics Series}.
\newblock U.S. Government Printing Office, Washington, D.C., 1964.

\bibitem{McIntoshOpHam}
D.~Albrecht, X.~Duong, and A.~McIntosh.
\newblock Operator theory and harmonic analysis.
\newblock In {\em Instructional {W}orkshop on {A}nalysis and {G}eometry, {P}art
  {III} ({C}anberra, 1995)}, volume~34 of {\em Proc. Centre Math. Appl.
  Austral. Nat. Univ.}, pages 77--136. Austral. Nat. Univ., Canberra, 1996.

\bibitem{BaillonClement91}
J.-B. Baillon and P.~Cl{\'e}ment.
\newblock Examples of unbounded imaginary powers of operators.
\newblock {\em J. Funct. Anal.}, 100(2):419--434, 1991.

\bibitem{battyhaasemubeen}
C.~Batty, M.~Haase, and J.~Mubeen.
\newblock The holomorphic functional calculus approach to operator semigroups.
\newblock {\em Acta Sci. Math. (Szeged)}, 79:289--323, 2013.

\bibitem{BenNik99}
N.-E. Benamara and N.~Nikolski.
\newblock Resolvent tests for similarity to a normal operator.
\newblock {\em Proc. London Math. Soc. (3)}, 78(3):585--626, 1999.

\bibitem{BoasEntireFunctions}
R.~P. Boas, Jr.
\newblock {\em Entire functions}.
\newblock Academic Press Inc., New York, 1954.

\bibitem{cowlingdoustmcintoshyagi}
M.~Cowling, I.~Doust, A.~McIntosh, and A.~Yagi.
\newblock Banach space operators with a bounded {$H^\infty$} functional
  calculus.
\newblock {\em J. Austral. Math. Soc. Ser. A}, 60(1):51--89, 1996.

\bibitem{EgertRozendaal13}
M.~Egert and J.~Rozendaal.
\newblock Convergence of subdiagonal {P}ad\'e approximations of
  {$C_0$}-semigroups.
\newblock {\em J. Evol. Equ.}, 13(4):875--895, 2013.

\bibitem{EisnerZwart06}
T.~Eisner and H.~Zwart.
\newblock Continuous-time {K}reiss resolvent condition on infinite-dimensional
  spaces.
\newblock {\em Math. Comp.}, 75(256):1971--1985, 2006.

\bibitem{EngelNagel}
K.-J. Engel and R.~Nagel.
\newblock {\em One-parameter semigroups for linear evolution equations}, volume
  194 of {\em Graduate Texts in Mathematics}.
\newblock Springer-Verlag, New York, 2000.

\bibitem{Fackler14Reg}
S.~Fackler.
\newblock Regularity {P}roperties of {S}ectorial {O}perators: {C}ounterexamples
  and {O}pen {P}roblems.
\newblock In W.~Arendt, R.~Chill, and Y.~Tomilov, editors, {\em Semigroups meet
  complex analysis, Harmonic Analysis and Mathematical Physics}, volume 250 of
  {\em Oper. Theory Adv. Appl.} Birkh\"auser, 2015.
\newblock %Preprint available at \url{http://arxiv.org/abs/1407.1142}.

\bibitem{Yakubovich2011}
J.~E. Gal{\'e}, P.~J. Miana, and D.~V. Yakubovich.
\newblock {$H^\infty$}-functional calculus and models of {N}agy-{F}oia\c s type
  for sectorial operators.
\newblock {\em Math. Ann.}, 351(3):733--760, 2011.

\bibitem{Garnett}
J.~B. Garnett.
\newblock {\em Bounded {A}nalytic {F}unctions}, volume 236 of {\em Graduate
  Texts in Mathematics}.
\newblock Springer, New York, first edition, 2007.

\bibitem{Gautschi59}
W.~Gautschi.
\newblock Some elementary inequalities relating to the gamma and incomplete
  gamma function.
\newblock {\em J. Math. and Phys.}, 38:77--81, 1959/60.

\bibitem{Gurarii71}
V.~I. Gurari{\u\i} and N.~I. Gurari{\u\i}.
\newblock Bases in uniformly convex and uniformly smooth {B}anach spaces.
\newblock {\em Izv. Akad. Nauk SSSR Ser. Mat.}, 35:210--215, 1971.

\bibitem{haakthesis}
B.~H. Haak.
\newblock {\em Kontrolltheorie in Banachr{\"a}umen und quadratische
  Absch{\"a}tzungen}.
\newblock {PhD} thesis, Universit\"at Karlsruhe, 2005.

\bibitem{Haak2012}
B.~H. Haak.
\newblock The {W}eiss conjecture and weak norms.
\newblock {\em J. Evol. Equ.}, 12(4):855--861, 2012.

\bibitem{haasesectorial}
M.~Haase.
\newblock {\em The {F}unctional {C}alculus for {S}ectorial {O}perators}, volume
  169 of {\em Operator Theory: Advances and Applications}.
\newblock Birkh\"auser Verlag, Basel, 2006.

\bibitem{haasehalfplaneoperators}
M.~Haase.
\newblock Semigroup theory via functional calculus.
\newblock Preprint available at:
  \url{http://fa.its.tudelft.nl/~haase/files/semi.pdf}, 2006.

\bibitem{haasetransference11}
M.~Haase.
\newblock Transference principles for semigroups and a theorem of {P}eller.
\newblock {\em J. Funct. Anal.}, 261(10):2959--2998, 2011.

\bibitem{HaaseRozendaal13}
M.~Haase and J.~Rozendaal.
\newblock Functional calculus for semigroup generators via transference.
\newblock {\em J. Funct. Anal.}, 265(12):3345--3368, 2013.

\bibitem{HavinJoericke}
V.~{Havin} and B.~{J\"oricke}.
\newblock {\em {The uncertainty principle in harmonic analysis.}}
\newblock Springer-Verlag, Berlin, 1994.

\bibitem{KaltonWeis01}
N.~J. Kalton and L.~Weis.
\newblock The {$H^\infty$}-calculus and sums of closed operators.
\newblock {\em Math. Ann.}, 321(2):319--345, 2001.

\bibitem{KunstmannWeis04}
P.~C. Kunstmann and L.~Weis.
\newblock Maximal {$L_p$}-regularity for parabolic equations, {F}ourier
  multiplier theorems and {$H^\infty$}-functional calculus.
\newblock In {\em Functional analytic methods for evolution equations}, volume
  1855 of {\em Lecture Notes in Math.}, pages 65--311. Springer, Berlin, 2004.

\bibitem{LeMerdy2003}
C.~Le~Merdy.
\newblock The {W}eiss conjecture for bounded analytic semigroups.
\newblock {\em J. London Math. Soc. (2)}, 67(3):715--738, 2003.

\bibitem{McCarthySchwartz65}
C.~A. McCarthy and J.~Schwartz.
\newblock On the norm of a finite {B}oolean algebra of projections, and
  applications to theorems of {K}reiss and {M}orton.
\newblock {\em Comm. Pure Appl. Math.}, 18:191--201, 1965.

\bibitem{mcintoshHinf}
A.~McIntosh.
\newblock Operators which have an {$H_\infty$} functional calculus.
\newblock In {\em Miniconference on operator theory and partial differential
  equations ({N}orth {R}yde, 1986)}, volume~14 of {\em Proc. Centre Math. Anal.
  Austral. Nat. Univ.}, pages 210--231. Austral. Nat. Univ., Canberra, 1986.

\bibitem{mcintoshHinfNo}
A.~McIntosh and A.~Yagi.
\newblock Operators of type {$\omega$} without a bounded {$H_\infty$}
  functional calculus.
\newblock In {\em Miniconference on {O}perators in {A}nalysis ({S}ydney,
  1989)}, volume~24 of {\em Proc. Centre Math. Anal. Austral. Nat. Univ.},
  pages 159--172. Austral. Nat. Univ., Canberra, 1990.

\bibitem{mubeenPhD}
J.~Mubeen.
\newblock {\em The bounded $\mathcal{H}^{\infty}$-calculus for sectorial,
  strip-type and half-plane operators}.
\newblock {PhD} thesis, University of Oxford, 2011.

\bibitem{Nikolski14}
N.~Nikolski.
\newblock Sublinear dimension growth in the {K}reiss matrix theorem.
\newblock {\em Algebra i Analiz}, 25(3):3--51, 2013.

\bibitem{Pazy83}
A.~Pazy.
\newblock {\em Semigroups of linear operators and applications to partial
  differential equations}, volume~44 of {\em Applied Mathematical Sciences}.
\newblock Springer-Verlag, New York, 1983.

\bibitem{Petermichl07}
S.~Petermichl.
\newblock The sharp bound for the {H}ilbert transform on weighted {L}ebesgue
  spaces in terms of the classical {$A_p$} characteristic.
\newblock {\em Amer. J. Math.}, 129(5):1355--1375, 2007.

\bibitem{Ritt53}
R.~K. Ritt.
\newblock A condition that {$\lim_{n\to\infty}n^{-1}T^n=0$}.
\newblock {\em Proc. Amer. Math. Soc.}, 4:898--899, 1953.

\bibitem{SchwenningerTR}
F.~L. Schwenninger.
\newblock {F}unctional calculus estimates for {T}admor--{R}itt operators.
\newblock{\em J. Math. Anal. Appl.}, 439(1):103--124, 2016.

\bibitem{SchweZwa2012}
F.~L. Schwenninger and H.~Zwart.
\newblock Weakly admissible {$\mathcal{H}_\infty^-$}-calculus on reflexive
  {B}anach spaces.
\newblock {\em Indag. Math. (N.S.)}, 23(4):796--815, 2012.

\bibitem{SingerBases}
I.~Singer.
\newblock {\em Bases in {B}anach spaces. {I}}.
\newblock Springer-Verlag, New York, 1970.
\newblock Die Grundlehren der mathematischen Wissenschaften, Band 154.

\bibitem{Spijker91}
M.~N. Spijker.
\newblock On a conjecture by {L}e{V}eque and {T}refethen related to the
  {K}reiss matrix theorem.
\newblock {\em BIT}, 31(3):551--555, 1991.

\bibitem{SpijkerTracognaWelfert03}
M.~N. Spijker, S.~Tracogna, and B.~D. Welfert.
\newblock About the sharpness of the stability estimates in the {K}reiss matrix
  theorem.
\newblock {\em Math. Comp.}, 72(242):697--713 (electronic), 2003.

\bibitem{TadmorLAA}
E.~Tadmor.
\newblock The resolvent condition and uniform power boundedness.
\newblock {\em Linear Algebra Appl.}, 80:250--252, 1986.

\bibitem{Vitse04}
P.~Vitse.
\newblock Functional calculus under the {T}admor-{R}itt condition, and free
  interpolation by polynomials of a given degree.
\newblock {\em J. Funct. Anal.}, 210(1):43--72, 2004.

\bibitem{Vitse2005b}
P.~Vitse.
\newblock A band limited and {B}esov class functional calculus for
  {T}admor-{R}itt operators.
\newblock {\em Arch. Math. (Basel)}, 85(4):374--385, 2005.

\bibitem{Vitse05}
P.~Vitse.
\newblock A {B}esov class functional calculus for bounded holomorphic
  semigroups.
\newblock {\em J. Funct. Anal.}, 228(2):245--269, 2005.

\bibitem{vonneumann}
J.~von Neumann.
\newblock {\em Mathematical foundations of quantum mechanics}.
\newblock Princeton Landmarks in Mathematics. Princeton University Press,
  Princeton, NJ, 1996.

\bibitem{ZwartAdmissible}
H.~Zwart.
\newblock Toeplitz operators and $\mathcal{H}^{\infty}$-calculus.
\newblock {\em J. Funct. Anal.}, 263(1):167 -- 182, 2012.

\end{thebibliography}

\end{document}